\documentclass{elsarticle}

\usepackage{lineno}
\modulolinenumbers[5]

\usepackage{ amsmath, amssymb, amsthm }
\usepackage{ array }
\usepackage{ booktabs }
\usepackage{ color }
\usepackage{ mdframed }
\usepackage{ qtree }
\usepackage{ url }

\theoremstyle{definition}

\newtheorem{definition}{Definition}[section]

\newtheorem{example}[definition]{Example}
\newtheorem{remark}[definition]{Remark}
\newtheorem*{algorithm}{Algorithm}

\theoremstyle{plain}

\newtheorem{lemma}[definition]{Lemma}
\newtheorem{proposition}[definition]{Proposition}
\newtheorem{theorem}[definition]{Theorem}

\renewcommand{\arg}{{\ast}}

\newcommand{\calI}{\mathcal{I}}
\newcommand{\calO}{\mathcal{O}}
\newcommand{\calP}{\mathcal{P}}
\newcommand{\calQ}{\mathcal{Q}}

\allowdisplaybreaks

\begin{document}

\begin{frontmatter}

\title{Quadratic Nonsymmetric Quaternary Operads}

\author[]{Murray R. Bremner\corref{mycorrespondingauthor}}

\cortext[mycorrespondingauthor]{Corresponding author.}

\address{Department of Mathematics and Statistics, University of Saskatchewan, Canada}

\ead{bremner@math.usask.ca}

\author{Juana S\'anchez-Ortega}

\address{Department of Mathematics and Applied Mathematics, University of Cape Town,\\South Africa}

\ead{juana.sanchez-ortega@uct.ac.za}

\begin{abstract}
We use computational linear algebra and commutative algebra to study spaces of relations satisfied
by quadrilinear operations.
The relations are analogues of associativity in the sense that they are quadratic (every term involves
two operations) and nonsymmetric (every term involves the identity permutation of the arguments).
We focus on determining those quadratic relations whose cubic consequences have minimal or maximal rank.
We approach these problems from the point of view of the theory of algebraic operads.
\end{abstract}


\end{frontmatter}

\linenumbers


\section{Introduction}

Let $\mathbb{F}$ be an algebraically closed field of characteristic 0.
We recall the definition of the free nonsymmetric operad $\calO$ over $\mathbb{F}$ generated by an 
$n$-ary operation $\phi$.
The operation $\phi$ has $n$ arguments, and any composition of $\phi$ with itself $w$ times is an 
operation whose arity $k$ (number of arguments) is given by
  \[
  k = 1 + w(n-1).
  \]
Thus self-compositions of an $n$-ary operation $\phi$ exist in arity $k$ if and only if $k$ is congruent 
to 1 modulo $n{-}1$.

These self-compositions are most naturally represented as tree monomials, where the trees are complete 
rooted planar $n$-ary trees.
For $w = 1$, the generator $\phi$ corresponds to the tree with one internal node and $n$ leaf nodes:
  \[
  \Tree [ .$\phi$ $\arg$ $\cdots$ $\arg$ ]
  \]
For $w = 2$ we obtain trees of the following form with $2n-1$ leaf nodes:
  \[
  \Tree [ .$\phi$ $\arg$ $\cdots$ [ .$\phi$ $\arg$ $\cdots$ $\arg$ ] $\cdots$ $\arg$ ]
  \]  
For $w = 3$ we obtain trees of the following two forms with $3n-2$ leaf nodes:
  \begin{center}
  \footnotesize
  \Tree [ .$\phi$ $\cdots$ [ .$\phi$ $\cdots$ [ .$\phi$ $\arg$ $\cdots$ $\arg$ ] $\cdots$ ] $\dots$ ] 
  \qquad
  \Tree [ .$\phi$ $\cdots$ [ .$\phi$ $\arg$ $\cdots$ $\arg$ ] $\cdots$ [ .$\phi$ $\arg$ $\cdots$ $\arg$ ] $\cdots$ ] 
  \end{center}

We impose a total order on tree monomials as follows: 
to compare two monomials, we find the position of the leftmost unequal subtrees; 
if these subtrees have different arities, the higher arity precedes; 
if these subtrees have the same arity, we recursively compare the subtrees.

Tree monomials can be composed in many different ways.
If $\alpha$ and $\beta$ are tree monomials of arities $k$ and $\ell$ respectively, then for $1 \le i \le k$
we define the composition $\alpha \circ_i \beta$ to be the tree monomial of arity $k{+}\ell{-}1$ obtained 
by substituting (the root of) $\beta$ for the $i$-th argument of $\alpha$.
We have:
  \begin{itemize}
  \item
  $\calO(1) = \mathbb{F} \iota$ where $\iota$ is the identity operation,
  \item
  $\calO(n) = \mathbb{F} \phi$,
  \item
  $\dim \calO(2n-1) = n$ with ordered basis $\{ \, \phi \circ_i \phi \mid 1 \le i \le n \, \}$,
  \item
  $\dim \calO(3n-2) = \frac12 n (3n-1)$ with ordered basis
    \[
    \{ \, \phi \circ_i ( \phi \circ_j \phi ) \mid 1 \le i, j \le n \, \} \; \cup \;
    \{ \, ( \phi \circ_j \phi ) \circ_i \phi \mid 1 \le i < j \le n \, \}.
    \]
  \end{itemize}

\begin{definition}
We write $\calO(k)$ for the vector space with basis consisting of all $n$-ary tree monomials of arity $k$.
The \textbf{free nonsymmetric $n$-ary operad} $\calO$ generated by the operation $\phi$ is the disjoint union
of the spaces $\calO(k)$ with the operations $\alpha \circ_i \beta$ extended bilinearly:
  \[
  \calO = \coprod_{w=0}^\infty \calO\big(1+w(n{-}1)\big).
  \]
\end{definition}

\begin{lemma}
The dimension of the space $\calO(k)$ is the $n$-ary Catalan number:
  \[
  \dim \calO\big(1+w(n{-}1)\big) = \frac{1}{(n-1)w+1} \binom{nw}{w}.
  \]
\end{lemma}

\begin{proof}
See \S7.5 in Graham et al.~\cite{GKP}.
\end{proof}

\begin{remark}
We assume that the operation $\phi$ has homological degree 0, so that we do not need to concern ourselves
with possible sign changes that may occur when compositions are performed in different orders.
For further details on this and other issues, see the comprehensive introduction to algebraic operads by Loday 
and Vallette \cite{LV}.
The algorithmic and computational aspects are discussed by Bremner and Dotsenko \cite{BD}.
\end{remark}

\begin{definition}  
If $w = 2$ then $k = 2n-1$, and any subspace $R \subseteq \calO(2n-1)$ is a \textbf{space of quadratic 
relations}, since every monomial involves two occurrences of the $n$-ary operation.
The \textbf{operad ideal} $\calI = (R)$ generated by the quadratic subspace $R \subseteq \calO(2n-1)$ is 
the smallest subspace $\calI \subseteq \calO$ which contains $R$ and is closed under left and right 
compositions by arbitrary elements of $\calO$.
That is, for all integers $k, \ell \ge 1$, all $\alpha \in \calI(k) = \calI \cap \calO(k)$, and all 
$\beta \in \calO(\ell)$, we have
  \[
  \alpha \circ_i \beta, \; \beta \circ_j \alpha \in \calI(k+\ell-1) 
  \qquad 
  (1 \le i \le k, 1 \le j \le \ell).
  \]
The \textbf{quotient operad} $\calQ = \calO/\calI$ is defined by $\calQ(k) = \calO(k)/\calI(k)$ for all 
$k \ge 1$ with the induced operations;
this is the general quadratic nonsymmetric $n$-ary operad with one generator.
\end{definition}


\section{The quaternary case ($n = 4$)}

For the rest of this paper, we consider only the case of a quaternary operation $(n = 4)$.
See \cite[Chapter 10]{BD} for a detailed study of the ternary case ($n = 3$).

In low arities, we can write out the tree monomials as placements of parentheses in sequences of
asterisks representing arguments to the operations.
For $\calO(1)$ we have $\{ \, \ast \, \}$;
for $\calO(4)$ we have $\{ \, \phi = (\arg\arg\arg\arg) \, \}$;
for $\calO(7)$ we have
  \[
  \{ \,
  \phi \circ_i \phi \mid 1 \le i \le 4  
  \, \}
  =
  \{ \,
  ((\arg\arg\arg\arg)\arg\arg\arg), \,
  (\arg(\arg\arg\arg\arg)\arg\arg), \, 
  (\arg\arg(\arg\arg\arg\arg)\arg), \,
  (\arg\arg\arg(\arg\arg\arg\arg)) 
  \, \}.
  \]
For $\calO(10)$ see Table \ref{arity10basis}.

  \begin{table}[ht]
  \[
  \boxed{
  \begin{array}{r@{\;}l@{\quad}r@{\;}l}
   1\colon & \phi \circ_1 ( \phi \circ_1 \phi ) = (((\arg\arg\arg\arg)\arg\arg\arg)\arg\arg\arg) &
   2\colon & \phi \circ_1 ( \phi \circ_2 \phi ) = ((\arg(\arg\arg\arg\arg)\arg\arg)\arg\arg\arg) \\
   3\colon & \phi \circ_1 ( \phi \circ_3 \phi ) = ((\arg\arg(\arg\arg\arg\arg)\arg)\arg\arg\arg) &
   4\colon & \phi \circ_1 ( \phi \circ_4 \phi ) = ((\arg\arg\arg(\arg\arg\arg\arg))\arg\arg\arg) \\
   5\colon & ( \phi \circ_2 \phi ) \circ_1 \phi = ((\arg\arg\arg\arg)(\arg\arg\arg\arg)\arg\arg) &
   6\colon & ( \phi \circ_3 \phi ) \circ_1 \phi = ((\arg\arg\arg\arg)\arg(\arg\arg\arg\arg)\arg) \\
   7\colon & ( \phi \circ_4 \phi ) \circ_1 \phi = ((\arg\arg\arg\arg)\arg\arg(\arg\arg\arg\arg)) &
   8\colon & \phi \circ_2 ( \phi \circ_1 \phi ) = (\arg((\arg\arg\arg\arg)\arg\arg\arg)\arg\arg) \\
   9\colon & \phi \circ_2 ( \phi \circ_2 \phi ) = (\arg(\arg(\arg\arg\arg\arg)\arg\arg)\arg\arg) &
  10\colon & \phi \circ_2 ( \phi \circ_3 \phi ) = (\arg(\arg\arg(\arg\arg\arg\arg)\arg)\arg\arg) \\
  11\colon & \phi \circ_2 ( \phi \circ_4 \phi ) = (\arg(\arg\arg\arg(\arg\arg\arg\arg))\arg\arg) &
  12\colon & ( \phi \circ_3 \phi ) \circ_2 \phi = (\arg(\arg\arg\arg\arg)(\arg\arg\arg\arg)\arg) \\
  13\colon & ( \phi \circ_4 \phi ) \circ_2 \phi = (\arg(\arg\arg\arg\arg)\arg(\arg\arg\arg\arg)) &
  14\colon & \phi \circ_3 ( \phi \circ_1 \phi ) = (\arg\arg((\arg\arg\arg\arg)\arg\arg\arg)\arg) \\
  15\colon & \phi \circ_3 ( \phi \circ_2 \phi ) = (\arg\arg(\arg(\arg\arg\arg\arg)\arg\arg)\arg) &
  16\colon & \phi \circ_3 ( \phi \circ_3 \phi ) = (\arg\arg(\arg\arg(\arg\arg\arg\arg)\arg)\arg) \\
  17\colon & \phi \circ_3 ( \phi \circ_4 \phi ) = (\arg\arg(\arg\arg\arg(\arg\arg\arg\arg))\arg) &
  18\colon & ( \phi \circ_4 \phi ) \circ_3 \phi = (\arg\arg(\arg\arg\arg\arg)(\arg\arg\arg\arg)) \\
  19\colon & \phi \circ_4 ( \phi \circ_1 \phi ) = (\arg\arg\arg((\arg\arg\arg\arg)\arg\arg\arg)) &
  20\colon & \phi \circ_4 ( \phi \circ_2 \phi ) = (\arg\arg\arg(\arg(\arg\arg\arg\arg)\arg\arg)) \\
  21\colon & \phi \circ_4 ( \phi \circ_3 \phi ) = (\arg\arg\arg(\arg\arg(\arg\arg\arg\arg)\arg)) &
  22\colon & \phi \circ_4 ( \phi \circ_4 \phi ) = (\arg\arg\arg(\arg\arg\arg(\arg\arg\arg\arg))) 
  \end{array}
  }
  \]
  \vspace{-4mm}
  \caption{The 22 tree monomials forming an ordered basis of $\calO(10)$}
  \label{arity10basis}
  \end{table}

Let $R \subseteq \calO(7)$ be a subspace of quadratic relations.
Given our ordered basis of $\calO(7)$, any subspace $R$ of dimension $r$ is the row space of 
a unique $r \times 4$ matrix $[R]$ in row canonical form (RCF).
\begin{itemize}
\item
For $r = 0$ we have $R = \{0\}$, there are no relations, and $\calQ \cong \calO$.
\item
For $r = 4$ we have $R = \calO(7)$, every quadratic tree monomial is zero, and $\calQ$ is nilpotent of index 2.
\item
For $0 < r < 4$ the relation matrices $[R]$ are given in Table \ref{table[R]}, where
$a, b, c, d$ are parameters that take arbitrary values in $\mathbb{F}$.
\end{itemize}
For each $r = 0,\dots,4$ there are $\binom{4}{r}$ choices for the columns of the leading 1s.

\begin{table}[ht]
\begin{mdframed}
\begin{alignat*}{5}
& r &&&&&&&&
\\
& 1 &\quad
&
\left[ \begin{array}{cccc} 
1 & a & b & c 
\end{array} \right]
&&
\left[ \begin{array}{cccc} 
0 & 1 & a & b 
\end{array} \right]
&&
\left[ \begin{array}{cccc} 
0 & 0 & 1 & a 
\end{array} \right]
&&
\left[ \begin{array}{cccc} 
0 & 0 & 0 & 1 
\end{array} \right]
\\
& 2 &\quad
&
\left[ \begin{array}{cccc} 
1 & 0 & a & b \\
0 & 1 & c & d 
\end{array} \right]
&&
\left[ \begin{array}{cccc} 
1 & a & 0 & b \\
0 & 0 & 1 & c 
\end{array} \right]
&&
\left[ \begin{array}{cccc} 
1 & a & b & 0 \\
0 & 0 & 0 & 1 
\end{array} \right]
&&
\\
&&&&&
\left[ \begin{array}{cccc} 
0 & 1 & 0 & a \\
0 & 0 & 1 & b 
\end{array} \right]
&&
\left[ \begin{array}{cccc} 
0 & 1 & a & 0 \\
0 & 0 & 0 & 1 
\end{array} \right]
&&
\left[ \begin{array}{cccc} 
0 & 0 & 1 & 0 \\
0 & 0 & 0 & 1 
\end{array} \right]
\\
& 3 &\quad
&
\left[ \begin{array}{cccc} 
1 & 0 & 0 & a \\
0 & 1 & 0 & b \\
0 & 0 & 1 & c
\end{array} \right]
&&
\left[ \begin{array}{cccc} 
1 & 0 & a & 0 \\
0 & 1 & b & 0 \\
0 & 0 & 0 & 1
\end{array} \right]
&&
\left[ \begin{array}{cccc} 
1 & a & 0 & 0 \\
0 & 0 & 1 & 0 \\
0 & 0 & 0 & 1
\end{array} \right]
&&
\left[ \begin{array}{cccc} 
0 & 1 & 0 & 0 \\
0 & 0 & 1 & 0 \\
0 & 0 & 0 & 1
\end{array} \right]
\end{alignat*}
\end{mdframed}
\caption{The relation matrices $[R]$ of ranks $1 \le r \le 3$ in arity 7}
\label{table[R]}
\end{table}

The entries of the matrix $[R]$ form an initial sequence of $t \ge 2$ elements of the ordered set 
$\{ 0, 1, a, b, c, d \}$, and except in a few special cases we have $t \ge 3$.
Thus in general $[R]$ is a matrix over a polynomial ring in 0 to 4 variables.
For matrices over $\mathbb{F}$ we can use Gaussian elimination to compute the RCF, and for matrices
over $\mathbb{F}[a]$ we can combine Gaussian elimination with the Euclidean algorithm for GCDs to
compute the Hermite normal form (HNF).
But for two or more variables, we have to use Gr\"obner basis techniques.
Let $P$ denote the polynomial ring over $\mathbb{F}$ in $k \ge 2$ variables.
We regard a row vector of length $t$ with entries in $P$ as an element of the free $P$-module $P^t$ of rank $t$.
We regard an $s \times t$ matrix $A$ with entries in $P$ as a set of $s$ generators for a submodule of $P^t$.
To find a canonical form for $A$, we may then appeal to the theory of Gr\"obner bases for submodules of free
modules over polynomial rings; see for example Cox et al. \cite[Chapter 5]{Cox}.

To compute the Gr\"obner basis for the submodule generated by the rows of $A$,
we use elementary row operations over $P$:
  \begin{itemize}
  \item
  interchange two rows,
  \item
  multiply a row by an invertible element of $P$ (a nonzero element of $\mathbb{F}$),
  \item
  add an arbitrary $P$-multiple of one row to another row. 
  \end{itemize} 
In order to have room to compute S-polynomials, we may also need to add zero rows to the bottom of the matrix.

We have introduced parameters into the quadratic relation matrices $[R]$, so we have implicitly 
changed the ring of coefficients from the field $\mathbb{F}$ to the polynomial ring $P$ generated over
$\mathbb{F}$ by the entries of $[R]$.
Thus all linear structures are no longer vector spaces over $\mathbb{F}$ but modules over $P$.
In particular $\calO(k)$ is a free $P$-module, and $\calI(k)$ is a submodule of this free $P$-module.

\begin{definition}
Every row of the relation matrix $[R]$ is the coefficient vector of a quadratic relation 
$\rho = \Sigma_{i=1}^4 c_i \, \phi \circ_i \phi \in \calO(7)$.
The homogeneous submodule $\calI(10)$ is generated
by the \textbf{consequences} of the relations $\rho$; for each $\rho$ these are obtained by 
right- and left-multiplying $\rho$ by $\phi$.
Thus $\rho \circ_i \phi$ ($1 \le i \le 7$) and $\phi \circ_j \rho$ ($1 \le j \le 4$) give 11
consequences for each $\rho$.
\end{definition}

\begin{lemma}
Assume that the quadratic relation matrix $[R]$ has rank $r$ and entries in the polynomial ring $P$.
Then the cubic relation module $\calI(10)$ is the submodule of the free $P$-module of rank 22 generated
by the rows of the $11r \times 22$ matrix whose rows are the coefficient vectors of the consequences
of $[R]$.
\end{lemma}

Among our goals in this paper, we mention the following: 
\begin{itemize}
\item
Determine how $\dim\calI(10)$ depends on the parameters in $[R]$,
and in particular, find the minimum and maximum cubic dimension.
\item
The operads with the minimum dimension have the smallest spaces of cubic relations,
hence the largest quotient spaces; for the given quadratic rank, they are as close to free as possible.
\item
The operads with the maximum dimension have the smallest quotient spaces; they are as close to nilpotent 
as possible.
\end{itemize}

All computations were performed on a MacBook Pro using the Maple 18 packages ListTools, LinearAlgebra,
and Groebner.


\section{Quadratic Relation Rank 1}

\subsection{Case 1}

For rank 1, case 1, the quadratic relation matrix is
\[
[R]
=
\left[ 
\begin{array}{cccc} 
1 & a & b & c
\end{array} 
\right]
\]
The row module $\calI(7)$ of $[R]$ over $P = \mathbb{F}[a,b,c]$ is generated by the relation
  \[
  \rho = 
  \phi \circ_1 \phi + a \, \phi \circ_2 \phi + b \, \phi \circ_3 \phi + c \, \phi \circ_4 \phi.
  \]

\begin{lemma}
Using the basis of Table \ref{arity10basis}, the 11 consequences of $\rho$ in arity 10 are given 
in Table \ref{rank1case1consequences}.
The matrix whose row module $\calI(10)$ is generated by these consequences is given in Table
\ref{rank1case1arity10}; the rows have been sorted so that the leading 1s go from left to right
(we have written dot for zero to increase legibility).
\end{lemma}

\begin{table}[ht]
\[
\boxed{
\begin{array}{l}
\rho \circ_1 \phi = 
\phi \circ_1 ( \phi \circ_1 \phi ) + a \, 
( \phi \circ_2 \phi ) \circ_1 \phi + b \, 
( \phi \circ_3 \phi ) \circ_1 \phi + c \, 
( \phi \circ_4 \phi ) \circ_1 \phi,
\\
\rho \circ_2 \phi = 
\phi \circ_1 ( \phi \circ_2 \phi ) + a \, 
\phi \circ_2 ( \phi \circ_1 \phi ) + b \, 
( \phi \circ_3 \phi ) \circ_2 \phi + c \, 
( \phi \circ_4 \phi ) \circ_2 \phi,
\\
\rho \circ_3 \phi = 
\phi \circ_1 ( \phi \circ_3 \phi ) + a \, 
\phi \circ_2 ( \phi \circ_2 \phi ) + b \, 
\phi \circ_3 ( \phi \circ_1 \phi ) + c \, 
( \phi \circ_4 \phi ) \circ_3 \phi,
\\
\rho \circ_4 \phi = 
\phi \circ_1 ( \phi \circ_4 \phi ) + a \, 
\phi \circ_2 ( \phi \circ_3 \phi ) + b \, 
\phi \circ_3 ( \phi \circ_2 \phi ) + c \, 
\phi \circ_4 ( \phi \circ_1 \phi ),
\\
\rho \circ_5 \phi = 
( \phi \circ_2 \phi ) \circ_1 \phi + a \, 
\phi \circ_2 ( \phi \circ_4 \phi ) + b \, 
\phi \circ_3 ( \phi \circ_3 \phi ) + c \, 
\phi \circ_4 ( \phi \circ_2 \phi ),
\\
\rho \circ_6 \phi = 
( \phi \circ_3 \phi ) \circ_1 \phi + a \, 
( \phi \circ_3 \phi ) \circ_2 \phi + b \, 
\phi \circ_3 ( \phi \circ_4 \phi ) + c \, 
\phi \circ_4 ( \phi \circ_3 \phi ),
\\
\rho \circ_7 \phi = 
( \phi \circ_4 \phi ) \circ_1 \phi + a \, 
( \phi \circ_4 \phi ) \circ_2 \phi + b \, 
( \phi \circ_4 \phi ) \circ_3 \phi + c \, 
\phi \circ_4 ( \phi \circ_4 \phi ),
\\
\phi \circ_1 \rho = 
\phi \circ_1 ( \phi \circ_1 \phi ) + a \, 
\phi \circ_1 ( \phi \circ_2 \phi ) + b \, 
\phi \circ_1 ( \phi \circ_3 \phi ) + c \, 
\phi \circ_1 ( \phi \circ_4 \phi ),
\\
\phi \circ_2 \rho = 
\phi \circ_2 ( \phi \circ_1 \phi ) + a \, 
\phi \circ_2 ( \phi \circ_2 \phi ) + b \, 
\phi \circ_2 ( \phi \circ_3 \phi ) + c \, 
\phi \circ_2 ( \phi \circ_4 \phi ),
\\
\phi \circ_3 \rho = 
\phi \circ_3 ( \phi \circ_1 \phi ) + a \, 
\phi \circ_3 ( \phi \circ_2 \phi ) + b \, 
\phi \circ_3 ( \phi \circ_3 \phi ) + c \, 
\phi \circ_3 ( \phi \circ_4 \phi ),
\\
\phi \circ_4 \rho = 
\phi \circ_4 ( \phi \circ_1 \phi ) + a \, 
\phi \circ_4 ( \phi \circ_2 \phi ) + b \, 
\phi \circ_4 ( \phi \circ_3 \phi ) + c \, 
\phi \circ_4 ( \phi \circ_4 \phi ).
\end{array}
}
\]
\vspace{-4mm}
\caption{Rank 1, case 1: consequences in arity 10 of relation $\rho$}
\label{rank1case1consequences}
\[
\left[
\begin{array}{cccccccccccccccccccccc}
1 & a & b & c & . & . & . & . & . & . & . & . & . & . & . & . & . & . & . & . & . & . \\[-2pt]
1 & . & . & . & a & b & c & . & . & . & . & . & . & . & . & . & . & . & . & . & . & . \\[-2pt]
. & 1 & . & . & . & . & . & a & . & . & . & b & c & . & . & . & . & . & . & . & . & . \\[-2pt]
. & . & 1 & . & . & . & . & . & a & . & . & . & . & b & . & . & . & c & . & . & . & . \\[-2pt]
. & . & . & 1 & . & . & . & . & . & a & . & . & . & . & b & . & . & . & c & . & . & . \\[-2pt]
. & . & . & . & 1 & . & . & . & . & . & a & . & . & . & . & b & . & . & . & c & . & . \\[-2pt]
. & . & . & . & . & 1 & . & . & . & . & . & a & . & . & . & . & b & . & . & . & c & . \\[-2pt]
. & . & . & . & . & . & 1 & . & . & . & . & . & a & . & . & . & . & b & . & . & . & c \\[-2pt]
. & . & . & . & . & . & . & 1 & a & b & c & . & . & . & . & . & . & . & . & . & . & . \\[-2pt]
. & . & . & . & . & . & . & . & . & . & . & . & . & 1 & a & b & c & . & . & . & . & . \\[-2pt]
. & . & . & . & . & . & . & . & . & . & . & . & . & . & . & . & . & . & 1 & a & b & c
\end{array}
\right]
\]
\vspace{-4mm}
\caption{Rank 1, case 1: matrix of consequences in arity 10}
\label{rank1case1arity10}
\end{table}

\begin{proposition} \label{rank1case1arity10theorem}
The matrix of Table \ref{rank1case1arity10} has minimal rank 10 and maximal rank 11.
Rank 10 is achieved for the following values of $a, b, c$:
  \[
  (a,b,c) \, = \,  
  (0,0,0), \;
  (0,0,-1), \;
  (-1,1,-1), \;
  (\alpha,\alpha^2,-1) \qquad (\alpha^2-\alpha+1 = 0).
  \]
All other values of $a, b, c$ give rank 11.
\end{proposition}

\begin{proof}
We write $P = \mathbb{F}[a,b,c]$.
We compute a Gr\"obner basis for the submodule of the free $P$-module $P^{22}$ generated by the 11 rows
of the matrix in Table \ref{rank1case1arity10}.
This is straightforward since except at one step the set $\{1\}$ is a Gr\"obner basis for the ideal
generated by the matrix entries at and below the current pivot.

We swap the leading 1s in the first 8 columns up to the diagonal and use them to create an identity matrix
of size 8 in the upper left corner.
At the next step, the pivot is in position $(9,9)$ and contains $-a^3+ab$ with two 0s below it; we multiply
row 9 by $-1$ to make the leading entry monic, and use this leading entry to reduce the entry in row 1.
The last two pivots are leading 1s in columns 14 and 19, which we use to eliminate the nonzero entries in
those columns.

Thus every row has a leading 1 except for row 9, so the matrix has rank 10 or 11, and the rank is 10 if and
only if row 9 is zero.
The nonzero entries in row 9 from left to right are:
  \[
  a^3-ab, \, a^2b-ac, \, a^2c+a^2, \, ab^2-bc, \, b^3+ab, \, b^2c+b^2, \, ac^2+ac, \, bc^2+bc, \, c^3+c^2.
  \]
These polynomials generate an ideal $I \subset P$ with this deglex Gr\"obner basis:
  \[
  b^2-ca, \quad
  cb+a^2, \quad
  a(a^2-b), \quad
  a(ab-c), \quad
  a^2(c+1), \quad
  ac(c+1), \quad
  c^2(c+1).
  \]
A power of each parameter occurs as the leading monomial of an element of the Gr\"obner basis, 
and so the ideal is zero-dimensional (its zero set is finite).

From the last element we see that either $c = 0$ or $c = -1$.
If $c = 0$ then the basis reduces to $b^2$, $a^2$, $a(a^2-b)$, $a^2b$, $a^2$ and hence $a = b = 0$.
If $c = -1$ then the basis reduces to 
  \[
  f_1 = b^2+a, \quad
  f_2 = a^2-b, \quad
  f_3 = a(a^2-b), \quad
  f_4 = a(ab+1). \quad
  \]
We have $f_3 = a f_2$ and $f_4 = f_1 + b f_2$, so we need only $f_1$ and $f_2$.
Setting $b = a^2$ in $b^2 + a$ gives $a^4 + a = a(a+1)(a^2-a+1)$, giving the solutions $a = 0$, $-1$, 
or a primitive 6th root of unity.
This completes the proof.
\end{proof}

\subsection{Case 2}

For rank 1, case 2, the quadratic relation matrix is
\[
[R]
=
\left[ 
\begin{array}{cccc} 
0 & 1 & a & b
\end{array} 
\right]
\]
The row module $\calI(7)$ of $[R]$ over $P = \mathbb{F}[a,b]$ is generated by the relation
  \[
  \rho = 
  \phi \circ_2 \phi + a \, \phi \circ_3 \phi + b \, \phi \circ_4 \phi.
  \]

\begin{lemma}
Using the basis of Table \ref{arity10basis}, the 11 consequences of $\rho$ in arity 10 are given 
in Table \ref{rank1case2consequences}.
The matrix whose row module $\calI(10)$ over $P$ is generated by these consequences is given in Table
\ref{rank1case2arity10}.
\end{lemma}

\begin{table}[ht]
\[
\boxed{
\begin{array}{l}
\rho \circ_1 \phi =  
( \phi \circ_2 \phi ) \circ_1 \phi + a \, 
( \phi \circ_3 \phi ) \circ_1 \phi + b\, 
( \phi \circ_4 \phi ) \circ_1 \phi,
\\
\rho \circ_2 \phi = 
\phi \circ_2 ( \phi \circ_1 \phi ) + a \, 
( \phi \circ_3 \phi ) \circ_2 \phi + b \, 
( \phi \circ_4 \phi ) \circ_2 \phi,
\\
\rho \circ_3 \phi = 
\phi \circ_2 ( \phi \circ_2 \phi ) + a \, 
\phi \circ_3 ( \phi \circ_1 \phi ) + b \, 
( \phi \circ_4 \phi ) \circ_3 \phi,
\\
\rho \circ_4 \phi = 
\phi \circ_2 ( \phi \circ_3 \phi ) + a \, 
\phi \circ_3 ( \phi \circ_2 \phi ) + b \, 
\phi \circ_4 ( \phi \circ_1 \phi ),
\\
\rho \circ_5 \phi =  
\phi \circ_2 ( \phi \circ_4 \phi ) + a \, 
\phi \circ_3 ( \phi \circ_3 \phi ) + b \, 
\phi \circ_4 ( \phi \circ_2 \phi ),
\\
\rho \circ_6 \phi = 
( \phi \circ_3 \phi ) \circ_2 \phi + a \, 
\phi \circ_3 ( \phi \circ_4 \phi ) + b \, 
\phi \circ_4 ( \phi \circ_3 \phi ),
\\
\rho \circ_7 \phi =  
( \phi \circ_4 \phi ) \circ_2 \phi + a \, 
( \phi \circ_4 \phi ) \circ_3 \phi + b \, 
\phi \circ_4 ( \phi \circ_4 \phi ),
\\
\phi \circ_1 \rho =  
\phi \circ_1 ( \phi \circ_2 \phi ) + a \, 
\phi \circ_1 ( \phi \circ_3 \phi ) + b \, 
\phi \circ_1 ( \phi \circ_4 \phi ),
\\
\phi \circ_2 \rho = 
\phi \circ_2 ( \phi \circ_2 \phi ) + a \, 
\phi \circ_2 ( \phi \circ_3 \phi ) + b \, 
\phi \circ_2 ( \phi \circ_4 \phi ),
\\
\phi \circ_3 \rho = 
\phi \circ_3 ( \phi \circ_2 \phi ) + a \, 
\phi \circ_3 ( \phi \circ_3 \phi ) + b \, 
\phi \circ_3 ( \phi \circ_4 \phi ),
\\
\phi \circ_4 \rho = 
\phi \circ_4 ( \phi \circ_2 \phi ) + a \, 
\phi \circ_4 ( \phi \circ_3 \phi ) + b \, 
\phi \circ_4 ( \phi \circ_4 \phi ).
\end{array}
}
\]
\vspace{-3mm}
\caption{Rank 1, case 2: consequences in arity 10 of the relation $\rho$}
\label{rank1case2consequences}
\[
\left[
\begin{array}{cccccccccccccccccccccc}
. & 1 & a & b & . & . & . & . & . & . & . & . & . & . & . & . & . & . & . & . & . & . \\[-2pt] 
. & . & . & . & 1 & a & b & . & . & . & . & . & . & . & . & . & . & . & . & . & . & . \\[-2pt] 
. & . & . & . & . & . & . & 1 & . & . & . & a & b & . & . & . & . & . & . & . & . & . \\[-2pt] 
. & . & . & . & . & . & . & . & 1 & a & b & . & . & . & . & . & . & . & . & . & . & . \\[-2pt] 
. & . & . & . & . & . & . & . & 1 & . & . & . & . & a & . & . & . & b & . & . & . & . \\[-2pt] 
. & . & . & . & . & . & . & . & . & 1 & . & . & . & . & a & . & . & . & b & . & . & . \\[-2pt] 
. & . & . & . & . & . & . & . & . & . & 1 & . & . & . & . & a & . & . & . & b & . & . \\[-2pt] 
. & . & . & . & . & . & . & . & . & . & . & 1 & . & . & . & . & a & . & . & . & b & . \\[-2pt] 
. & . & . & . & . & . & . & . & . & . & . & . & 1 & . & . & . & . & a & . & . & . & b \\[-2pt] 
. & . & . & . & . & . & . & . & . & . & . & . & . & . & 1 & a & b & . & . & . & . & . \\[-2pt] 
. & . & . & . & . & . & . & . & . & . & . & . & . & . & . & . & . & . & . & 1 & a & b
\end{array}
\right]
\]
\vspace{-4mm}
\caption{Rank 1, case 2: matrix of consequences in arity 10}
\label{rank1case2arity10}
\end{table}

\begin{proposition}
The rank of the matrix of Table \ref{rank1case2arity10} is either 10 or 11; 
it is 10 if and only if $a = b = 0$.
\end{proposition}

\begin{proof}
We compute a Gr\"obner basis for the submodule of $P^{22}$ generated by the rows of the matrix 
of Table \ref{rank1case2arity10} following the proof of Proposition \ref{rank1case1arity10theorem}, 
and obtain the following canonical form of the matrix:
\[
\left[
\begin{array}{c@{\;\;}c@{\;\;}c@{\;\;}c@{\;\;}c@{\;\;}c@{\;\;}c@{\;\;}c@{\;\;}c@{\;\;}c@{\;\;}c@{\;\;}c@{\;\;}c@{\;\;}c@{\;\;}c@{\;\;}c@{\;\;}c@{\;\;}c@{\;\;}c@{\;\;}c@{\;\;}c@{\;\;}c}
. & 1 & a & b & . & . & . & . & . & . & . & . & . & . & . & . & . & . & . & . & . & . \\[-2pt]
. & . & . & . & 1 & a & b & . & . & . & . & . & . & . & . & . & . & . & . & . & . & . \\[-2pt]
. & . & . & . & . & . & . & 1 & . & . & . & . & . & . & . & . & {-}a^2 & {-}a b & . & . & {-}a b & {-}b^2 \\[-2pt]
. & . & . & . & . & . & . & . & 1 & . & . & . & . & . & . & a^3{-}a b & a^2 b & . & {-}a b & . & b^2 a & b^3 \\[-2pt]
. & . & . & . & . & . & . & . & . & 1 & . & . & . & . & . & {-}a^2 & {-}a b & . & b & . & . & . \\[-2pt]
. & . & . & . & . & . & . & . & . & . & 1 & . & . & . & . & a & . & . & . & . & {-}a b & {-}b^2 \\[-2pt]
. & . & . & . & . & . & . & . & . & . & . & 1 & . & . & . & . & a & . & . & . & b & . \\[-2pt]
. & . & . & . & . & . & . & . & . & . & . & . & 1 & . & . & . & . & a & . & . & . & b \\[-2pt]
. & . & . & . & . & . & . & . & . & . & . & . & . & a & . & {-}a^3{+}a b & {-}a^2 b & b & a b & . & {-}b^2 a & {-}b^3 \\[-2pt]
. & . & . & . & . & . & . & . & . & . & . & . & . & . & 1 & a & b & . & . & . & . & . \\[-2pt]
. & . & . & . & . & . & . & . & . & . & . & . & . & . & . & . & . & . & . & 1 & a & b
\end{array}
\right]
\]
The only row whose leading entry is not 1 is row 9, and the nonzero elements of row 9 are 
$a$, ${-}a^3{+}a b$, ${-}a^2 b$, $b$, $a b$, ${-}b^2 a$, ${-}b^3$, which generate the ideal whose 
deglex Gr\"obner basis is $\{ a, b \}$ and whose zero set is the single point $\{ (0,0) \}$.
This completes the proof.
\end{proof}

\subsection{Case 3}

For rank 1, case 3, the quadratic relation matrix is
\[
[R]
=
\left[ 
\begin{array}{cccc} 
0 & 0 & 1 & a
\end{array} 
\right]
\]
The row module $\calI(7)$ of $[R]$ over $P = \mathbb{F}[a]$ is generated by the relation
  \[
  \rho = 
  \phi \circ_3 \phi + a \, \phi \circ_4 \phi.
  \]
In this case $P$ is a PID so we compute the HNF of the matrix of consequences.

\begin{table}[ht]
\[
\left[
\begin{array}{cccccccccccccccccccccc}
. & . & 1 & a & . & . & . & . & . & . & . & . & . & . & . & . & . & . & . & . & . & . \\[-2pt] 
. & . & . & . & . & 1 & a & . & . & . & . & . & . & . & . & . & . & . & . & . & . & . \\[-2pt] 
. & . & . & . & . & . & . & . & . & 1 & a & . & . & . & . & . & . & . & . & . & . & . \\[-2pt] 
. & . & . & . & . & . & . & . & . & . & . & 1 & a & . & . & . & . & . & . & . & . & . \\[-2pt] 
. & . & . & . & . & . & . & . & . & . & . & . & . & 1 & . & . & . & a & . & . & . & . \\[-2pt] 
. & . & . & . & . & . & . & . & . & . & . & . & . & . & 1 & . & . & . & a & . & . & . \\[-2pt] 
. & . & . & . & . & . & . & . & . & . & . & . & . & . & . & 1 & a & . & . & . & . & . \\[-2pt] 
. & . & . & . & . & . & . & . & . & . & . & . & . & . & . & 1 & . & . & . & a & . & . \\[-2pt] 
. & . & . & . & . & . & . & . & . & . & . & . & . & . & . & . & 1 & . & . & . & a & . \\[-2pt] 
. & . & . & . & . & . & . & . & . & . & . & . & . & . & . & . & . & 1 & . & . & . & a \\[-2pt] 
. & . & . & . & . & . & . & . & . & . & . & . & . & . & . & . & . & . & . & . & 1 & a
\end{array}
\right]
\]
\vspace{-4mm}
\caption{Rank 1, case 3: matrix of consequences in arity 10}
\label{rank1case3arity10}
\end{table}

\begin{proposition}
The matrix whose row module $\calI(10)$ over $P$ is generated by the consequences of $[R]$ is given 
in Table \ref{rank1case3arity10}.
This matrix has rank either 10 or 11; the rank is 10 if and only if $a = 0$.
\end{proposition}

\begin{proof}
Column 16 contains two leading 1s which suggests that the matrix does not always have full rank.
The HNF of the matrix is as follows:
\[
\left[
\begin{array}{cccccccccccccccccccccc}
. & . & 1 & a & . & . & . & . & . & . & . & . & . & . & . & . & . & . & . & . & . & . \\[-2pt]
. & . & . & . & . & 1 & a & . & . & . & . & . & . & . & . & . & . & . & . & . & . & . \\[-2pt]
. & . & . & . & . & . & . & . & . & 1 & a & . & . & . & . & . & . & . & . & . & . & . \\[-2pt]
. & . & . & . & . & . & . & . & . & . & . & 1 & a & . & . & . & . & . & . & . & . & . \\[-2pt]
. & . & . & . & . & . & . & . & . & . & . & . & . & 1 & . & . & . & . & . & . & . & -a^2 \\[-2pt]
. & . & . & . & . & . & . & . & . & . & . & . & . & . & 1 & . & . & . & a & . & . & . \\[-2pt]
. & . & . & . & . & . & . & . & . & . & . & . & . & . & . & 1 & . & . & . & . & . & a^3 \\[-2pt]
. & . & . & . & . & . & . & . & . & . & . & . & . & . & . & . & 1 & . & . & . & . & -a^2 \\[-2pt]
. & . & . & . & . & . & . & . & . & . & . & . & . & . & . & . & . & 1 & . & . & . & a \\[-2pt]
. & . & . & . & . & . & . & . & . & . & . & . & . & . & . & . & . & . & . & a & . & -a^3 \\[-2pt]
. & . & . & . & . & . & . & . & . & . & . & . & . & . & . & . & . & . & . & . & 1 & a
\end{array}
\right]
\]
Row 10 is the only row which does not have a leading 1, and its nonzero entries are $a$, $-a^3$.
This completes the proof.
\end{proof}

\subsection{Case 4}

For rank 1, case 4, the quadratic relation matrix is
\[
[R]
=
\left[ 
\begin{array}{cccc} 
0 & 0 & 0 & 1
\end{array} 
\right]
\]
The row module $\calI(7)$ of $[R]$ over $P = \mathbb{F}$ is generated by the relation
  \[
  \rho = 
  \phi \circ_4 \phi.
  \]
In this case $P$ is a field so we can find the rank of the matrix of consequences by computing its RCF.

\begin{table}[h!]
\[
\left[
\begin{array}{cccccccccccccccccccccc}
. & . & . & 1 & . & . & . & . & . & . & . & . & . & . & . & . & . & . & . & . & . & . \\[-2pt] 
. & . & . & . & . & . & 1 & . & . & . & . & . & . & . & . & . & . & . & . & . & . & . \\[-2pt] 
. & . & . & . & . & . & . & . & . & . & 1 & . & . & . & . & . & . & . & . & . & . & . \\[-2pt] 
. & . & . & . & . & . & . & . & . & . & . & . & 1 & . & . & . & . & . & . & . & . & . \\[-2pt] 
. & . & . & . & . & . & . & . & . & . & . & . & . & . & . & . & 1 & . & . & . & . & . \\[-2pt] 
. & . & . & . & . & . & . & . & . & . & . & . & . & . & . & . & . & 1 & . & . & . & . \\[-2pt] 
. & . & . & . & . & . & . & . & . & . & . & . & . & . & . & . & . & . & 1 & . & . & . \\[-2pt] 
. & . & . & . & . & . & . & . & . & . & . & . & . & . & . & . & . & . & . & 1 & . & . \\[-2pt] 
. & . & . & . & . & . & . & . & . & . & . & . & . & . & . & . & . & . & . & . & 1 & . \\[-2pt] 
. & . & . & . & . & . & . & . & . & . & . & . & . & . & . & . & . & . & . & . & . & 1
\end{array}
\right]
\]
\vspace{-4mm}
\caption{Rank 1, case 4: matrix of consequences in arity 10}
\label{rank1case4arity10}
\end{table}

\begin{proposition}
Table \ref{rank1case4arity10} gives the matrix whose row space $\calI(10)$ over $P$ is the space of 
consequences in arity 10.
This matrix is in RCF and has rank 10.
\end{proposition}

\subsection{Summary for quadratic relation rank 1}

\begin{theorem}
For a 1-dimensional space of quadratic relations in arity 7, the rank of the cubic consequences 
in arity 10 is either 10 or 11.
None of the corresponding operads are nilpotent of index 3.
For the subspaces spanned by the following 8 relations, the rank is 10:
  \begin{align*}
  1a) \qquad 
  \rho &= 
  \phi \circ_1 \phi,
  \\
  1b) \qquad 
  \rho &= 
  \phi \circ_1 \phi - \phi \circ_4 \phi,
  \\
  1c) \qquad 
  \rho &= 
  \phi \circ_1 \phi - \phi \circ_2 \phi + \phi \circ_3 \phi - \phi \circ_4 \phi,
  \\
  1d) \qquad 
  \rho &= 
  \phi \circ_1 \phi + \alpha \, \phi \circ_2 \phi + \alpha^2 \, \phi \circ_3 \phi - \phi \circ_4 \phi
  \quad (\alpha^2 - \alpha + 1 = 0),
  \\
  2) \qquad 
  \rho &= 
  \phi \circ_2 \phi,
  \\
  3) \qquad 
  \rho &= 
  \phi \circ_3 \phi,
  \\
  4) \qquad 
  \rho &= 
  \phi \circ_4 \phi.
  \end{align*}
For all other 1-dimensional subspaces, the rank is 11.
\end{theorem}

\begin{remark}
Relation 1$c$) and its generalization to all even $n$ (the alternating sum over all quadratic monomials) was discovered 
by Gnedbaye \cite{Gnedbaye} in his study of Koszul duality for $n$-ary operads.
The analogous relation for odd $n$ is the sum over all quadratic monomials.
\end{remark}


\section{Quadratic Relation Rank 2}

\subsection{Case 1}

For rank 2, case 1, the quadratic relation matrix is
\[
[R]
=
\left[ 
\begin{array}{cccc} 
1 & 0 & a & b \\
0 & 1 & c & d 
\end{array} 
\right]
\]
Each row of this matrix has 11 cubic consequences which are linear combinations of the 22 monomials
in Table \ref{arity10basis}.

\begin{lemma}
The $22 \times 22$ matrix over $\mathbb{F}[a,b,c,d]$ whose rows generate the module of cubic consequences
of the quadratic relations $[R]$ is given in Table \ref{rank2case1arity10}; 
the rows have been sorted so the leading 1s go from left to right.
\end{lemma}

\begin{table}[ht]
\[
\left[
\begin{array}{cccccccccccccccccccccc}
1 & . & a & b & . & . & . & . & . & . & . & . & . & . & . & . & . & . & . & . & . & . \\[-2pt] 
1 & . & . & . & . & a & b & . & . & . & . & . & . & . & . & . & . & . & . & . & . & . \\[-2pt] 
. & 1 & c & d & . & . & . & . & . & . & . & . & . & . & . & . & . & . & . & . & . & . \\[-2pt] 
. & 1 & . & . & . & . & . & . & . & . & . & a & b & . & . & . & . & . & . & . & . & . \\[-2pt] 
. & . & 1 & . & . & . & . & . & . & . & . & . & . & a & . & . & . & b & . & . & . & . \\[-2pt] 
. & . & . & 1 & . & . & . & . & . & . & . & . & . & . & a & . & . & . & b & . & . & . \\[-2pt] 
. & . & . & . & 1 & c & d & . & . & . & . & . & . & . & . & . & . & . & . & . & . & . \\[-2pt] 
. & . & . & . & 1 & . & . & . & . & . & . & . & . & . & . & a & . & . & . & b & . & . \\[-2pt] 
. & . & . & . & . & 1 & . & . & . & . & . & . & . & . & . & . & a & . & . & . & b & . \\[-2pt] 
. & . & . & . & . & . & 1 & . & . & . & . & . & . & . & . & . & . & a & . & . & . & b \\[-2pt] 
. & . & . & . & . & . & . & 1 & . & a & b & . & . & . & . & . & . & . & . & . & . & . \\[-2pt] 
. & . & . & . & . & . & . & 1 & . & . & . & c & d & . & . & . & . & . & . & . & . & . \\[-2pt] 
. & . & . & . & . & . & . & . & 1 & c & d & . & . & . & . & . & . & . & . & . & . & . \\[-2pt] 
. & . & . & . & . & . & . & . & 1 & . & . & . & . & c & . & . & . & d & . & . & . & . \\[-2pt] 
. & . & . & . & . & . & . & . & . & 1 & . & . & . & . & c & . & . & . & d & . & . & . \\[-2pt] 
. & . & . & . & . & . & . & . & . & . & 1 & . & . & . & . & c & . & . & . & d & . & . \\[-2pt] 
. & . & . & . & . & . & . & . & . & . & . & 1 & . & . & . & . & c & . & . & . & d & . \\[-2pt] 
. & . & . & . & . & . & . & . & . & . & . & . & 1 & . & . & . & . & c & . & . & . & d \\[-2pt] 
. & . & . & . & . & . & . & . & . & . & . & . & . & 1 & . & a & b & . & . & . & . & . \\[-2pt] 
. & . & . & . & . & . & . & . & . & . & . & . & . & . & 1 & c & d & . & . & . & . & . \\[-2pt] 
. & . & . & . & . & . & . & . & . & . & . & . & . & . & . & . & . & . & 1 & . & a & b \\[-2pt] 
. & . & . & . & . & . & . & . & . & . & . & . & . & . & . & . & . & . & . & 1 & c & d
\end{array}
\right]
\]
\vspace{-4mm}
\caption{Rank 2, case 1: matrix of consequences in arity 10}
\label{rank2case1arity10}
\end{table}

In this case, to compute a reduced form of the matrix, we use the algorithm for the ``partial Smith
form'' discussed in \cite[Chapter 8]{BD} and reproduced below.
This algorithm uses elementary row and column operations to reduce an $m \times n$ matrix $A$ over a 
polynomial ring $P$ (with coefficients in $\mathbb{F}$) containing many nonzero scalar entries to a 
block diagonal matrix with an upper left identity matrix $I_r$ measuring the minimal rank $r$ of $A$ 
and a lower right block $B$ which is typically much smaller than $A$ and contains all the information 
needed to understand how the rank of $A$ depends on the parameters in $P$.
Row operations replace one generating set for the submodule generated by the $m$ rows of $A$ by another 
generating set.
Column operations replace one basis for the free $P$-module $P^n$ by another basis.
In both cases, the determinant of the transformation is a nonzero element of $\mathbb{F}$.
The output of this algorithm is not a canonical form of the matrix $A$, but it has the information we
need in a more compact form than the Gr\"obner basis for the submodule generated by the rows of $A$;
see Remark \ref{rank2case1remark} for further information about Gr\"obner bases for modules.

\begin{table}[ht]
\begin{algorithm} 
\emph{Input}: An $m \times n$ matrix $A = ( a_{ij} )$ with entries in $\mathbb{F}[x_1,\dots,x_k]$.
\emph{Output}: An $m \times n$ block diagonal matrix $C = \mathrm{diag}(I_r,B)$ such that:
  \begin{itemize}
  \item
  $C = UAV$ where $U$ ($m \times m$) and $V$ ($n \times n$) are products of elementary matrices over
  $\mathbb{F}[x_1,\dots,x_k]$ and hence $\det(U), \det(V) \in \mathbb{F} \setminus \{0\}$;
  in other words, $C$ is row-column equivalent to $A$.
  \item
  $I_r$ is the $r \times r$ identity matrix where $r$ is the minimal rank of $A$ over all values of 
  $x_1,\dots,x_k$:
    \[
    r = \min \{ \; \mathrm{rank} ( A|_{x_1=a_1,\dots,x_k=a_k} ) \mid (a_1,\dots,a_k) \in \mathbb{F}^k \; \}.
    \]
  \item
  $B$ is an $(m{-}r) \times (n{-}r)$ matrix over $\mathbb{F}[x_1,\dots,x_k]$ such that
    \[
    \mathrm{rank} ( A|_{x_1=a_1,\dots,x_k=a_k} ) = r + \mathrm{rank} ( B|_{x_1=a_1,\dots,x_k=a_k} ),
    \;
    \forall (a_1,\dots,a_k) \in \mathbb{F}^k.
    \]
  \end{itemize}
\begin{mdframed}
  \begin{enumerate}
  \item
  Set $r \leftarrow 1$.
  \item
  While some $a_{ij} \in \mathbb{F} \setminus \{0\}$ for $i, j \ge r$ do:
    \begin{enumerate}
    \item
    Set $i \leftarrow r$.
    While $a_{ij} \notin \mathbb{F} \setminus \{0\}$ for all $j \ge r$ do: set $i \leftarrow i + 1$.
    \item
    If $i \ne r$ then interchange rows $i$ and $r$.
    \item
    Set $j \leftarrow r$.
    While $a_{rj} \notin \mathbb{F} \setminus \{0\}$ do: set $j \leftarrow j + 1$.
    \item
    If $j \ne r$ then interchange columns $j$ and $r$.
    \item
    If $a_{rr} \ne 1$ then multiply row $r$ by $1/a_{rr}$. [\emph{Create new diagonal 1.}]
    \item
    For $i$ from $r+1$ to $m$ do:
      \begin{enumerate}
      \item[-]
      If $a_{ir} \ne 0$ then add $-a_{ir}$ times row $r$ to row $i$.
      \end{enumerate}
    \item
    For $j$ from $r+1$ to $n$ do:
      \begin{enumerate}
      \item[-]
      If $a_{rj} \ne 0$ then add $-a_{rj}$ times column $r$ to column $j$.
      \end{enumerate}
    \item
    Set $r \leftarrow r + 1$.
    \end{enumerate}
  \item
  Set $r \leftarrow r - 1$.
  \end{enumerate}
\end{mdframed}
\end{algorithm}
\vspace{-3mm}
\caption{Partial Smith form of a matrix over a polynomial ring}
\label{psftable}
\end{table}

\begin{lemma} \label{Blemma}
We use the algorithm in Table \ref{psftable} to compute the partial Smith form of the matrix of consequences
in Table \ref{rank2case1arity10}.
We obtain the block diagonal matrix $\mathrm{diag}( I_{17}, B )$ where
\[
-B =
\left[
\begin{array}{c@{\;\;}c@{\;\;}c@{\;\;}c@{\;\;}c}
0 & a^2c+acd & abc+ad^2+ac & abd+ad & b^2d+bd \\ 
-d & c^3+ac-cd & c^2d+bc & acd+cd^2 & bcd+d^3 \\ 
0 & a^3+abc & a^2b+abd+a^2 & ab^2+ab & b^3+b^2 \\ 
-ad & -a & -ac & 0 & 0 \\ 
cd & ac^2-bc & acd+c^2 & a^2d+bcd+cd & abd+bd^2+d^2
\end{array}
\right]
\]
\end{lemma}

\begin{proof}
The algorithm in Table \ref{psftable} requires $r = 17$ iterations and 87 row/column operations;
see Table \ref{87operations}.
The result is an upper left identity matrix of size 17 and (the negative of) the $5 \times 5$ 
lower right block displayed above.
\end{proof}

\begin{table}[ht]
\[
\boxed{
\begin{array}{r@{\;\;}l@{\;}l@{\;}l@{\;}l@{\;}l}
%
%
1) & 
R_{2} - R_{1}, & 
C_{3} -a C_{1}, & 
C_{4} -b C_{1},   
\\ 
%
%
2) &
R_{3} \leftrightarrow R_{2}, & 
R_{4} - R_{2}, & 
C_{3} -c C_{2}, & 
C_{4} -d C_{2},   
\\
%
%
3) &
R_{5} \leftrightarrow R_{3}, & 
R_{4} +c R_{3}, & 
R_{5} +a R_{3}, & 
C_{14} -a C_{3}, & 
C_{18} -b C_{3},   
\\ 
%
%
4) &
R_{6} \leftrightarrow R_{4}, & 
R_{5} +b R_{4}, & 
R_{6} +d R_{4}, & 
C_{15} -a C_{4}, & 
C_{19} -b C_{4},   
\\ 
%
%
5) &
R_{7} \leftrightarrow R_{5}, & 
R_{8} - R_{5}, & 
C_{6} -c C_{5}, & 
C_{7} -d C_{5},   
\\
%
%
6) &
R_{9} \leftrightarrow R_{6}, & 
R_{7} -a R_{6}, & 
R_{8} +c R_{6}, & 
C_{17} -a C_{6}, & 
C_{21} -b C_{6},   
\\
%
%
7) &
R_{10} \leftrightarrow R_{7}, & 
R_{8} +d R_{7}, & 
R_{10} -b R_{7}, & 
C_{18} -a C_{7}, & 
C_{22} -b C_{7},   
\\
%
%
8) &
R_{11} \leftrightarrow R_{8}, & 
R_{12} - R_{8}, & 
C_{10} -a C_{8}, & 
C_{11} -b C_{8},   
\\
%
%
9) &
R_{13} \leftrightarrow R_{9}, & 
R_{14} - R_{9}, & 
C_{10} -c C_{9}, & 
C_{11} -d C_{9},   
\\ 
%
%
10) &
R_{15} \leftrightarrow R_{10}, & 
R_{12} +a R_{10}, & 
R_{14} +c R_{10}, & 
C_{15} -c C_{10}, & 
C_{19} -d C_{10},   
\\
%
%
11) &
R_{16} \leftrightarrow R_{11}, & 
R_{12} +b R_{11}, & 
R_{14} +d R_{11}, & 
C_{16} -c C_{11}, & 
C_{20} -d C_{11},   
\\
%
%
12) &
R_{17} \leftrightarrow R_{12}, & 
R_{13} -a R_{12}, & 
R_{17} -c R_{12}, & 
C_{17} -c C_{12}, & 
C_{21} -d C_{12},   
\\
%
%
13) &
R_{18} \leftrightarrow R_{13}, & 
R_{17} -d R_{13}, & 
R_{18} -b R_{13}, & 
C_{18} -c C_{13}, & 
C_{22} -d C_{13},   
\\
%
%
14) &
R_{19} \leftrightarrow R_{14}, & 
R_{15} -a^2 R_{14}, & 
R_{18} -ac R_{14}, & 
R_{19} -c R_{14}, & 
C_{16} -a C_{14},   
\\
& 
C_{17} -b C_{14},   
\\
%
%
15) &
R_{20} \leftrightarrow R_{15}, & 
R_{17} -ac R_{15}, & 
R_{18} -ad R_{15}, & 
R_{19} -c^2 R_{15}, & 
R_{20} -ab R_{15},   
\\
&
C_{16} -c C_{15}, & 
C_{17} -d C_{15},   
\\
%
%
16) &
R_{21} \leftrightarrow R_{16}, & 
C_{19} \leftrightarrow C_{16}, & 
R_{17} -ad R_{16}, & 
R_{18} -bd R_{16}, & 
R_{19} -cd R_{16},   
\\
&
R_{20} -b^2 R_{16}, & 
C_{21} -a C_{16}, & 
C_{22} -b C_{16},   
\\
%
%
17) &
R_{22} \leftrightarrow R_{17}, & 
C_{20} \leftrightarrow C_{17}, & 
R_{19} -d^2 R_{17}, & 
R_{21} -b R_{17}, & 
R_{22} -bd R_{17},   
\\
& 
C_{21} -c C_{17}, & 
C_{22} -d C_{17}.   
\end{array}
}
\]
\vspace{-4mm}
\caption{Row and column operations for proof of Lemma \ref{Blemma}}
\label{87operations}
\end{table}

\begin{proposition}
The only values of the parameters $(a,b,c,d)$ for which the matrix $B$ of Lemma \ref{Blemma} equals zero 
(and hence for which the matrix of consequences of Table \ref{rank2case1arity10} has minimal rank 17) 
are $(0,0,0,0)$ and $(0,-1,0,0)$.
\end{proposition}

\begin{proof}
The deglex Gr\"obner basis ($a \prec b \prec c \prec d$) 
for the ideal generated by the entries of $B$ is
$\{ a, d, bc, c^2, b^2(b+1) \}$ from which it follows immediately that $a = c = d = 0$ and $b \in \{ 0, -1 \}$.
\end{proof}

\begin{remark} \label{rank2case1remark}
Instead of computing the partial Smith form of the matrix in Table \ref{rank2case1arity10}, we could instead
compute its ``row canonical form'' by finding the Gr\"obner basis for the submodule generated by its rows.
However, we will see that this requires much more work and does not give us any more useful information.

For the first 15 columns, there is always an entry 1 at or below the pivot, so we easily create an identity
matrix of size 15 in the upper left corner.
When we reach column 16, the nonzero entries at or below the pivot are
  \[
  a, \quad -ac^2+bc, \quad -a^3-abc, \quad -a^2c-acd, \quad -c^3-ac+cd.
  \]
These entries generate an ideal with deglex Gr\"obner basis $\{ \, a, bc, c^3-cd \, \}$; this is easily
obtained from the generators using row operations since the first generator $a$ allows us to 
eliminate every term containing $a$ from the other generators.

The pivot is now at $(19,17)$; there are two nonzero entries at or below it:
  \[
  f = c^2ba+ca^3-dba-ba^2-a^2, \qquad g = dc^2a+c^2a^2-d^2a-cba-ca.
  \]
These entries generate an ideal whose deglex Gr\"obnex basis contains $f, g$ and
  \begin{align*}
  h &= dca^3+ca^4+cb^2a-2dba^2-ba^3+cba-da^2-a^3,
  \\
  k &= 2dcba^2-d^2a^3+cba^3-db^2a+dca^2-b^2a^2+ca^3-dba-2ba^2-a^2,
  \\
  \ell &= d^2a^4+cba^4+2cb^3a-3db^2a^2-b^2a^3+3cb^2a-3dba^2-ba^3+cba-da^2.
  \end{align*}
The existence of non-trivial S-polynomials means that to compute the Gr\"obner basis using row operations
we need to add zero rows to the matrix.
For example, $f$ and $g$ produce the S-polynomial $s = df - bg$ with reduced form $s' = s + af = h$.
In matrix terms, we create a row 23, perform $R_{23} + d R_{19}$ and $R_{23} - b R_{20}$ to create $s$
in row 23, perform $R_{23} + a R_{19}$ to create $h$ in row 23, and then do $R_{21} \leftrightarrow R_{23}$
to move $h$ to row 21, so that now $f, g, h$ occur in rows 19-21.
Further similar calculations produce the complete Gr\"obner basis for column 17; since the
Gr\"obner basis has five elements but the original generating set has only two, we need to add three
rows to the matrix.

After processing column 17 it remains only to deal with two more leading 1s in columns 19 and 20, and the
computation is complete.
The result is a $25 \times 22$ matrix; after deleting the rows and columns containing the leading 1s, we
are left with an $8 \times 5$ submatrix (analogous to the block $B$ of Lemma \ref{Blemma}).
\end{remark}
  
\begin{remark}
The ideal generated by the entries of $B$ is the first determinantal ideal of $B$.
In general, the $r$-th determinantal ideal of $B$, denoted $DI_r(B)$, is the ideal
generated by the $r \times r$ minors of $B$.
We write $V(DI_r(B))$ for the zero set of this ideal: all points in the parameter space $\mathbb{F}^k$
for which every $f \in DI_r(B)$ vanishes. 
Then $\mathrm{rank}(B) = r$ for $(a_1,\dots,a_k) \in \mathbb{F}^k$ if and only if 
$(a_1,\dots,a_k) \in V(DI_{r+1}(B)) \setminus V(DI_r(B))$.
\end{remark}

\begin{example} \label{rank2case1example}
The second determinantal ideal $DI_2(B)$ is generated by 87 distinct nonzero monic $2 \times 2$ minors
of degrees 4, 5, 6 with up to 11 terms and coefficients $\pm 1, \pm 2$.
Its deglex Gr\"obner basis has 18 elements of degrees 2 to 5 with up to 4 terms and coefficients $\pm 1$:
  \[
  \begin{array}{l}
  ad, \quad abc, \quad bcd, \quad ac(a^2{+}c), \quad a^2b(b{+}1), \quad a^2(c^2{+}ab{+}a), \quad 
  ab^2(b{+}1), 
  \\[1mm]
  ac(c^2{+}a), \quad b^2d(b{+}1), \quad bd^2(b{+}1), \quad cd(c^2{-}d), \quad d^2(c^2{-}bd{-}d), \quad 
  cd(d^2{+}c), 
  \\[1mm]
  a^2(a^2b{+}a^2{+}c), \quad b^3c(b{+}1), \quad b^2c^2(b{+}1), \quad d^2(bd^2{+}d^2{+}c), \quad 
  c(c^4{+}b^2c{-}a^2{-}d^2).
  \end{array}
  \]
The ideal $DI_2(B)$ is not zero-dimensional; removing the two points in $V(DI_1(B))$ from $V(DI_2(B))$, 
we are left with a 1-parameter family and 6 isolated points:
  \begin{align*}
  &
  V(DI_2(B)) \setminus V(DI_1(B))
  =
  \{ \, (0,-1,0,d) \mid d \in \mathbb{F}, \, d \ne 0 \, \} \, \cup \,
  \{ \,
  (-1,0,-1,0),  
  \\
  &
  (0,-1,-1,0), \,
  (0,0,-1,1), \,
  (0,-1,\alpha,0), \,
  (\alpha,0,\alpha^{-1},0), \,
  (0,0,\alpha,\alpha^2)
  \, \},
  \end{align*}
where $\alpha$ is a primitive 6th root of unity ($\alpha^2-\alpha+1=0$).
All of these parameter values produce rank exactly 18 for the matrix of consequences in arity 10.
\end{example}

\begin{proposition} \label{rank2case1nilpotentproposition}
The determinant of the block $B$ of Lemma \ref{Blemma} factors as follows,
where $\alpha$ is a primitive 6th root of unity:
  \begin{align*}
  \det(B)
  &=
  - a^2 d^2 (ad-bc)^2 (ad-bc-b-c-1)
  \times {}
  \\
  &\qquad
  \big( ad - bc - c + \alpha( b + 1 ) \big)
  \big( ad - bc - c + \alpha^{-1}( b + 1 ) \big).
  \end{align*}
For any quadruple $(a,b,c,d)$ of parameter values satisfying $\det(B) \ne 0$, the rank of the matrix of
consequences (Table \ref{rank2case1arity10}) is 22 which is $\dim\calO(10)$, and so the corresponding 
operads are nilpotent of index 3.
These operads form a Zariski dense subset of the parameter space $\mathbb{F}^4$.
For any quadruple $(a,b,c,d)$ satisfying $\det(B) = 0$, the rank of the matrix of consequences will be 
strictly less than 22.
(These operads may still be nilpotent, but of index strictly greater than 3.)
\end{proposition}

\begin{proof}
Since $B$ is a square matrix, its top determinantal ideal $DI_5(B)$ is the principal ideal generated by
its determinant, and $\mathrm{rank}(B) < 5$ if and only if $a,b,c,d$
satisfy $\det(B) = 0$.
Since $17 + 5 = 22 = \dim \calO(10)$, the rest follows.
\end{proof}

\begin{remark}
Proposition \ref{rank2case1nilpotentproposition} illustrates the general fact that whenever the maximal 
rank of the matrix of consequences in any weight is equal to the dimension of the free operad in that 
weight, and that is the lowest weight for which this equality occurs, then a Zariski dense subset of the 
parameter values correspond to operads which are nilpotent of index equal to that weight.
\end{remark}

\subsection{Case 2}

For rank 2, case 2, the quadratic relation matrix is
\[
\left[ 
\begin{array}{cccc} 
1 & a & 0 & b \\
0 & 0 & 1 & c 
\end{array} 
\right]
\]
Each row has 11 cubic consequences which are linear combinations of the 22 monomials
in Table \ref{arity10basis}.
The $22 \times 22$ matrix over $\mathbb{F}[a,b,c]$ whose rows generate the module of cubic consequences
is given in Table \ref{rank2case2arity10}.

\begin{table}[ht]
\[
\left[
\begin{array}{cccccccccccccccccccccc}
1 & a & . & b & . & . & . & . & . & . & . & . & . & . & . & . & . & . & . & . & . & . \\[-2pt] 
1 & . & . & . & a & . & b & . & . & . & . & . & . & . & . & . & . & . & . & . & . & . \\[-2pt] 
. & 1 & . & . & . & . & . & a & . & . & . & . & b & . & . & . & . & . & . & . & . & . \\[-2pt] 
. & . & 1 & c & . & . & . & . & . & . & . & . & . & . & . & . & . & . & . & . & . & . \\[-2pt] 
. & . & 1 & . & . & . & . & . & a & . & . & . & . & . & . & . & . & b & . & . & . & . \\[-2pt] 
. & . & . & 1 & . & . & . & . & . & a & . & . & . & . & . & . & . & . & b & . & . & . \\[-2pt] 
. & . & . & . & 1 & . & . & . & . & . & a & . & . & . & . & . & . & . & . & b & . & . \\[-2pt] 
. & . & . & . & . & 1 & c & . & . & . & . & . & . & . & . & . & . & . & . & . & . & . \\[-2pt] 
. & . & . & . & . & 1 & . & . & . & . & . & a & . & . & . & . & . & . & . & . & b & . \\[-2pt] 
. & . & . & . & . & . & 1 & . & . & . & . & . & a & . & . & . & . & . & . & . & . & b \\[-2pt] 
. & . & . & . & . & . & . & 1 & a & . & b & . & . & . & . & . & . & . & . & . & . & . \\[-2pt] 
. & . & . & . & . & . & . & . & . & 1 & c & . & . & . & . & . & . & . & . & . & . & . \\[-2pt] 
. & . & . & . & . & . & . & . & . & . & . & 1 & c & . & . & . & . & . & . & . & . & . \\[-2pt] 
. & . & . & . & . & . & . & . & . & . & . & . & . & 1 & a & . & b & . & . & . & . & . \\[-2pt] 
. & . & . & . & . & . & . & . & . & . & . & . & . & 1 & . & . & . & c & . & . & . & . \\[-2pt] 
. & . & . & . & . & . & . & . & . & . & . & . & . & . & 1 & . & . & . & c & . & . & . \\[-2pt] 
. & . & . & . & . & . & . & . & . & . & . & . & . & . & . & 1 & c & . & . & . & . & . \\[-2pt] 
. & . & . & . & . & . & . & . & . & . & . & . & . & . & . & 1 & . & . & . & c & . & . \\[-2pt] 
. & . & . & . & . & . & . & . & . & . & . & . & . & . & . & . & 1 & . & . & . & c & . \\[-2pt] 
. & . & . & . & . & . & . & . & . & . & . & . & . & . & . & . & . & 1 & . & . & . & c \\[-2pt] 
. & . & . & . & . & . & . & . & . & . & . & . & . & . & . & . & . & . & 1 & a & . & b \\[-2pt] 
. & . & . & . & . & . & . & . & . & . & . & . & . & . & . & . & . & . & . & . & 1 & c
\end{array}
\right]
\]
\vspace{-4mm}
\caption{Rank 2, case 2: matrix of consequences in arity 10}
\label{rank2case2arity10}
\end{table}

\begin{proposition} \label{rank2case2minrrank}
The only values of the parameters $(a,b,c)$ for which the matrix of consequences of Table
\ref{rank2case2arity10} has minimal rank 17 are $(0,0,0)$ and $(0,-1,0)$.
\end{proposition}

\begin{proof}
In this case it is not difficult to compute the Gr\"obner basis for the submodule generated by the rows; 
at every step, the ideal generated by the entries at or below the pivot is principal.
The final result is presented in Table \ref{rank2case2arity10grobnerbasis}; the last row of the
matrix is zero, which shows that the matrix never attains full rank and that the corresponding
operads are never nilpotent of index 3.

If we delete the 17 rows and 17 columns of the Gr\"obner basis which contain leading 1s, then we obtain 
this $5 \times 5$ block, which is a canonical analogue of the lower right block obtained from the algorithm 
for the partial Smith form:
\begin{equation}
\label{rank2case2lrb}
\left[
\begin{array}{ccccc}
a & {-}ac^2 & \cdot & \cdot & {-}abc^3{-}b^2c{-}bc \\ 
\cdot & a^3c^2{+}a^2b{+}abc{+}a^2 & \cdot & ab^2{+}ab & {-}ab^2c^2{+}a^2bc{-}abc^2{+}b^3{+}b^2 \\ 
\cdot & \cdot & \cdot & c & {-}c^3 \\ 
\cdot & \cdot & \cdot & \cdot & a^2c^3{+}abc{+}bc^2{+}c^2 \\ 
\cdot & \cdot & \cdot & \cdot & \cdot
\end{array}
\right]
\end{equation}
(Row 5 and column 3 are zero.)
The deglex Gr\"obner basis for the ideal of $\mathbb{F}[a,b,c]$ generated by the entries of this matrix is
$\{ a, c, b^2(b+1) \}$ which shows that the rank is 17 if and only if $a = c = 0$ and $b \in \{ 0, -1 \}$.
\end{proof}

\begin{table}[ht]
\[
\left[
\begin{array}{@{\;}c@{\;}c@{\;}c@{\;}c@{\;}c@{\;}c@{\;}c@{\;}c@{\;}c@{\;}c@{\;}c@{\;}c@{\;}c@{\;}c@{\;}c@{\;}c@{\;}c@{\;}c@{\;}c@{\;}c@{\;}c@{\;}c@{\;}}
1 & \cdot & \cdot & \cdot & \cdot & \cdot & \cdot & \cdot & \cdot & \cdot & {-}a^2 & \cdot & {-}ab & \cdot & \cdot & \cdot & \cdot & \cdot & \cdot & {-}ab & \cdot & {-}b^2 \\ 
\cdot & 1 & \cdot & \cdot & \cdot & \cdot & \cdot & \cdot & \cdot & \cdot & {-}a^2c^2{-}ab & \cdot & b & \cdot & \cdot & \cdot & \cdot & \cdot & \cdot & \cdot & \cdot & b^2c^2{-}abc{+}bc^2 \\ 
\cdot & \cdot & 1 & \cdot & \cdot & \cdot & \cdot & \cdot & \cdot & \cdot & ac^2 & \cdot & \cdot & \cdot & \cdot & \cdot & \cdot & \cdot & \cdot & \cdot & \cdot & abc^3{+}b^2c \\ 
\cdot & \cdot & \cdot & 1 & \cdot & \cdot & \cdot & \cdot & \cdot & \cdot & {-}ca & \cdot & \cdot & \cdot & \cdot & \cdot & \cdot & \cdot & \cdot & {-}ab & \cdot & {-}b^2 \\ 
\cdot & \cdot & \cdot & \cdot & 1 & \cdot & \cdot & \cdot & \cdot & \cdot & a & \cdot & \cdot & \cdot & \cdot & \cdot & \cdot & \cdot & \cdot & b & \cdot & \cdot \\ 
\cdot & \cdot & \cdot & \cdot & \cdot & 1 & \cdot & \cdot & \cdot & \cdot & \cdot & \cdot & {-}ca & \cdot & \cdot & \cdot & \cdot & \cdot & \cdot & \cdot & \cdot & {-}bc \\ 
\cdot & \cdot & \cdot & \cdot & \cdot & \cdot & 1 & \cdot & \cdot & \cdot & \cdot & \cdot & a & \cdot & \cdot & \cdot & \cdot & \cdot & \cdot & \cdot & \cdot & b \\ 
\cdot & \cdot & \cdot & \cdot & \cdot & \cdot & \cdot & 1 & \cdot & \cdot & ac^2{+}b & \cdot & \cdot & \cdot & \cdot & \cdot & \cdot & \cdot & \cdot & \cdot & \cdot & abc^3{+}b^2c{+}bc \\ 
\cdot & \cdot & \cdot & \cdot & \cdot & \cdot & \cdot & \cdot & a & \cdot & {-}ac^2 & \cdot & \cdot & \cdot & \cdot & \cdot & \cdot & \cdot & \cdot & \cdot & \cdot & {-}abc^3{-}b^2c{-}bc \\ 
\cdot & \cdot & \cdot & \cdot & \cdot & \cdot & \cdot & \cdot & \cdot & 1 & c & \cdot & \cdot & \cdot & \cdot & \cdot & \cdot & \cdot & \cdot & \cdot & \cdot & \cdot \\ 
\cdot & \cdot & \cdot & \cdot & \cdot & \cdot & \cdot & \cdot & \cdot & \cdot & 
\begin{array}{c} a^3c^2{+}a^2b \\ {+}abc{+}a^2 \end{array} 
& \cdot & \cdot & \cdot & \cdot & \cdot & \cdot & \cdot & \cdot & ab^2{+}ab & \cdot & 
\begin{array}{c} {-}ab^2c^2{+}a^2bc \\ {-}abc^2{+}b^3{+}b^2 \end{array} 
\\ 
\cdot & \cdot & \cdot & \cdot & \cdot & \cdot & \cdot & \cdot & \cdot & \cdot & \cdot & 1 & c & \cdot & \cdot & \cdot & \cdot & \cdot & \cdot & \cdot & \cdot & \cdot \\ 
\cdot & \cdot & \cdot & \cdot & \cdot & \cdot & \cdot & \cdot & \cdot & \cdot & \cdot & \cdot & \cdot & 1 & \cdot & \cdot & \cdot & \cdot & \cdot & \cdot & \cdot & {-}c^2 \\ 
\cdot & \cdot & \cdot & \cdot & \cdot & \cdot & \cdot & \cdot & \cdot & \cdot & \cdot & \cdot & \cdot & \cdot & 1 & \cdot & \cdot & \cdot & \cdot & \cdot & \cdot & {-}ac^3{-}bc \\ 
\cdot & \cdot & \cdot & \cdot & \cdot & \cdot & \cdot & \cdot & \cdot & \cdot & \cdot & \cdot & \cdot & \cdot & \cdot & 1 & \cdot & \cdot & \cdot & \cdot & \cdot & c^3 \\ 
\cdot & \cdot & \cdot & \cdot & \cdot & \cdot & \cdot & \cdot & \cdot & \cdot & \cdot & \cdot & \cdot & \cdot & \cdot & \cdot & 1 & \cdot & \cdot & \cdot & \cdot & {-}c^2 \\ 
\cdot & \cdot & \cdot & \cdot & \cdot & \cdot & \cdot & \cdot & \cdot & \cdot & \cdot & \cdot & \cdot & \cdot & \cdot & \cdot & \cdot & 1 & \cdot & \cdot & \cdot & c \\ 
\cdot & \cdot & \cdot & \cdot & \cdot & \cdot & \cdot & \cdot & \cdot & \cdot & \cdot & \cdot & \cdot & \cdot & \cdot & \cdot & \cdot & \cdot & 1 & a & \cdot & b \\ 
\cdot & \cdot & \cdot & \cdot & \cdot & \cdot & \cdot & \cdot & \cdot & \cdot & \cdot & \cdot & \cdot & \cdot & \cdot & \cdot & \cdot & \cdot & \cdot & c & \cdot & {-}c^3 \\ 
\cdot & \cdot & \cdot & \cdot & \cdot & \cdot & \cdot & \cdot & \cdot & \cdot & \cdot & \cdot & \cdot & \cdot & \cdot & \cdot & \cdot & \cdot & \cdot & \cdot & 1 & c \\ 
\cdot & \cdot & \cdot & \cdot & \cdot & \cdot & \cdot & \cdot & \cdot & \cdot & \cdot & \cdot & \cdot & \cdot & \cdot & \cdot & \cdot & \cdot & \cdot & \cdot & \cdot & a^2c^3{+}abc{+}bc^2{+}c^2 \\ 
\cdot & \cdot & \cdot & \cdot & \cdot & \cdot & \cdot & \cdot & \cdot & \cdot & \cdot & \cdot & \cdot & \cdot & \cdot & \cdot & \cdot & \cdot & \cdot & \cdot & \cdot & \cdot
\end{array}
\right]
\]
\vspace{-4mm}
\caption{Rank 2, case 2: deglex Gr\"obner basis for submodule generated by consequences}
\label{rank2case2arity10grobnerbasis}
\end{table}

\begin{example}
As in Example \ref{rank2case1example}, we get more information about the rank of the matrix \eqref{rank2case2lrb} 
from its determinantal ideals.
The nonzero $2 \times 2$ minors are
  \begin{align*}
  &
  ac, \quad a^3c, \quad ac(ab{+}bc{+}a), \quad ac(ab{+}bc{+}c), \quad ac^3, \quad ab^2c(b{+}1), \quad a^5c,
  \\
  &
  a^2b(a^2c{+}b{+}1), \quad a^2(a^2c^2{+}ab{+}bc{+}a), \quad ab(b{+}1)(a^2c{+}b), \quad a^2bc(ab{+}a{-}c),
  \\
  &
  a^3c(ab{+}bc{+}a), \quad a^3c(ab{+}bc{+}c), \quad a^3c^3, \quad a^2bc(b{+}1)^2, \quad c^2(a^2c^2{+}ab{+}bc{+}c),
  \\
  &
  abc^2(b{+}1)^2, \quad bc(b{+}1)(ac^2{+}b), \quad {-}abc^2({-}bc{+}a{-}c), \quad ac(ab{+}bc{+}a)(ab{+}bc{+}c),
  \\
  &
  bc^2(ac^2{+}b{+}1), \quad ac^3(ab{+}bc{+}a), \quad ac^3(ab{+}bc{+}c), \quad ac^5.
  \end{align*}
The deglex Gr\"obner basis for the second determinantal ideal is as follows:
  \[
  ac, \quad a^3(b+1), \quad a^2b(b+1), \quad ab^2(b+1), \quad b^2c(b+1), \quad bc^2(b+1), \quad (b+1)c^3.
  \]
The zero set of this ideal is the union of three lines in $\mathbb{F}^3$:
  \[
  \{ \, (a,-1,0) \mid a \in \mathbb{F} \, \} \; \cup \;
  \{ \, (0,b,0) \mid b \in \mathbb{F} \, \} \; \cup \;
  \{ \, (0,-1,c) \mid c \in \mathbb{F} \, \}.
  \]  
Removing the two points of Proposition \ref{rank2case2minrrank}, we obtain the parameter values which produce a
matrix of consequences with rank exactly 18.
The two points being removed are the origin and the intersection of the three lines:
  \[
  \{ (a,-1,0) \mid a \in \mathbb{F}, a \ne 0 \} \cup
  \{ (0,b,0) \mid b \in \mathbb{F}, b \ne 0, -1 \} \cup
  \{ (0,-1,c) \mid c \in \mathbb{F}, c \ne 0 \}.
  \]  
These computations could be extended to higher determinantal ideals, but the results rapidly become more complicated.
\end{example}

\subsection{Case 3}
For rank 2, case 3, the quadratic relation matrix is
\[
\left[ 
\begin{array}{cccc} 
1 & a & b & 0 \\
0 & 0 & 0 & 1 
\end{array} 
\right]
\]

\begin{proposition}
For rank 2, case 3, the only parameter values which produce the minimal rank 17 for the matrix of 
consequences in arity 10 (not displayed) are $a = b = 0$.
Rank 18 occurs only for two pairs: $(a,b) = (-1,1)$, $(\alpha,\alpha^2)$ where $\alpha^2-\alpha+1=0$.
All other values of $a$ and $b$ give rank 19.
\end{proposition}

\begin{proof}
The matrix of consequences has size $21 \times 22$ since the quadratic monomial relation produces the
same cubic relation in two different ways.
Its partial Smith form has two zero rows and one zero column, which we delete; the remaining $19 \times 21$
matrix consists of an identity matrix of size 17 and this $2 \times 4$ block:
  \[
  \left[
  \begin{array}{cccc}
  -a & 0 & -b & 0 \\
  -a^2 b & -a^3+a b & -b^2 a & -b^3-a b
  \end{array}
  \right]
  \]
Clearly the ideal generated by the entries of the block has Gr\"obner basis $\{ a, b \}$.
The deglex Gr\"obner basis for the ideal generated by the $2 \times 2$ minors is
  \[
  a^4-a^2b, \quad a^3b-ab^2, \quad ab^3+a^2b, \quad b^4+ab^2.
  \]  
This is a zero-dimensional ideal with zero set
  \[
  (a,b) = (0,0), \; (-1,1), \; (\alpha,\alpha^2) \qquad (\alpha^2-\alpha+1=0).
  \]
This completes the proof.
\end{proof}

\subsection{Case 4}
For rank 2, case 4, the quadratic relation matrix is
\[
\left[ 
\begin{array}{cccc} 
0 & 1 & 0 & a \\
0 & 0 & 1 & b 
\end{array} 
\right]
\]

\begin{proposition}
For rank 2, case 4, the only parameter values which produce the minimal rank 17 for the matrix of 
consequences in arity 10 (not displayed) are $a = b = 0$.
Rank 18 occurs only for two pairs: $(a,b) = (-1,-1)$, $(\alpha,\alpha^{-1})$ where $\alpha^2-\alpha+1=0$.
All other values of $a$ and $b$ give rank 19.
\end{proposition}

\begin{proof}
The matrix of consequences has size $22 \times 22$, and its partial Smith form has one zero row and 
three zero columns, which we delete. 
The remaining $21 \times 19$ matrix consists of an identity matrix of size 17 and the following 
$4 \times 2$ block:
  \[
  \left[
  \begin{array}{cc}
  0 & -b^3-a b \\
  a & -b a^2 \\
  b & -b^2 a \\
  0 & -a^3-a b
  \end{array}
  \right]
  \]
The ideal generated by the entries has Gr\"obner basis $\{ a, b \}$.
The deglex Gr\"obner basis for the second determinantal ideal is the same as Case 3
up to sign changes:
  \[  
  a^4+a^2b, \quad a^3b+ab^2, \quad ab^3+a^2b, \quad b^4+ab^2.
  \]  
This is a zero-dimensional ideal with zero set
  \[
  (a,b) = (0,0), \; (-1,-1), \; (\alpha,\alpha^{-1}) \qquad (\alpha^2-\alpha+1=0).
  \]
This completes the proof.
\end{proof}

\subsection{Case 5}
For rank 2, case 5, the quadratic relation matrix is
\[
\left[ 
\begin{array}{cccc} 
0 & 1 & a & 0 \\
0 & 0 & 0 & 1 
\end{array} 
\right]
\]

\begin{proposition}
For rank 2, case 5, the only parameter value which produces the minimal rank of 17 for the matrix of 
consequences in arity 10 (not displayed) is $a = 0$; the rank is 19 if $a \ne 0$.
\end{proposition}

\begin{proof}
The matrix of consequences has size $21 \times 22$, and contains only one parameter, so we can compute
its Hermite normal form.
The HNF has two zero rows and one zero column, which we delete, leaving a matrix of size $19 \times 21$.
There are 17 rows with leadings 1s; the other two rows have $a$ as the leading (and only nonzero) entry.
This completes the proof.
\end{proof}

\subsection{Case 6}
For rank 2, case 6, the quadratic relation matrix is
\[
\left[ 
\begin{array}{cccc} 
0 & 0 & 1 & 0 \\
0 & 0 & 0 & 1 
\end{array} 
\right]
\]

\begin{proposition}
For rank 2, case 6, there are no parameters and the matrix of consequences has rank 17.
\end{proposition}

\begin{proof}
The matrix of consequences has only 17 rows since the two quadratic monomial relations produce duplicate
cubic consequences.
The $17 \times 22$ matrix has no zero rows and five zero columns, and is already in row canonical form.
\end{proof}

\subsection{Summary for quadratic relation rank 2}

\begin{theorem}
For a 2-dimensional space of quadratic relations in arity 7, the rank $r$ of the cubic consequences satisfies
$17 \le r \le 22$.
In case 1, a Zariski dense subset of the space of parameter values corresponds to rank 22 and hence to operads 
which are nilpotent of index 3; this happens in no other case.
The following is a complete list of pairs of quadratic relations for which the rank achieves the minimal value of 17:
  \begin{align*}
  1a) \qquad 
  \rho_1 &= 
  \phi \circ_1 \phi,
  \qquad
  \rho_2 =
  \phi \circ_2 \phi,
  \\
  1b) \qquad 
  \rho_1 &= 
  \phi \circ_1 \phi - \phi \circ_4 \phi,
  \qquad
  \rho_2 =
  \phi \circ_2 \phi,  
  \\
  2a) \qquad 
  \rho_1 &= 
  \phi \circ_1 \phi,
  \qquad
  \rho_2 =
  \phi \circ_3 \phi,  
  \\
  2b) \qquad 
  \rho_1 &= 
  \phi \circ_1 \phi - \phi \circ_4 \phi,
  \qquad
  \rho_2 =
  \phi \circ_3 \phi,  
  \\
  3) \qquad 
  \rho_1 &= 
  \phi \circ_1 \phi,
  \qquad
  \rho_2 =
  \phi \circ_4 \phi,  
  \\
  4) \qquad 
  \rho_1 &= 
  \phi \circ_2 \phi,
  \qquad
  \rho_2 =
  \phi \circ_3 \phi,  
  \\
  5) \qquad 
  \rho_1 &= 
  \phi \circ_2 \phi,
  \qquad
  \rho_2 =
  \phi \circ_4 \phi,  
  \\
  6) \qquad 
  \rho_1 &= 
  \phi \circ_3 \phi,
  \qquad
  \rho_2 =
  \phi \circ_4 \phi.
  \end{align*}
\end{theorem}


\section{Quadratic Relation Rank 3}

\subsection{Case 1}

For rank 3, case 1, the quadratic relation matrix is
\[
\left[ \begin{array}{cccc} 
1 & 0 & 0 & a \\
0 & 1 & 0 & b \\
0 & 0 & 1 & c
\end{array} \right]
\]
Each row of this matrix is the coefficient vector of a quadratic relation which has 11 cubic consequences 
which are linear combinations of the 22 monomials in Table \ref{arity10basis}.
The resulting $33 \times 22$ matrix over $P = \mathbb{F}[a,b,c]$ whose rows generate the submodule of the free 
module $P^{22}$ of cubic consequences is given in Table 
\ref{rank3case1arity10}; the rows have been sorted so the leading 1s go from left to right.

\begin{table}[ht]
\[
\left[
\begin{array}{cccccccccccccccccccccc}
1 & . & . & a & . & . & . & . & . & . & . & . & . & . & . & . & . & . & . & . & . & . \\[-2pt] 
1 & . & . & . & . & . & a & . & . & . & . & . & . & . & . & . & . & . & . & . & . & . \\[-2pt] 
. & 1 & . & b & . & . & . & . & . & . & . & . & . & . & . & . & . & . & . & . & . & . \\[-2pt] 
. & 1 & . & . & . & . & . & . & . & . & . & . & a & . & . & . & . & . & . & . & . & . \\[-2pt] 
. & . & 1 & c & . & . & . & . & . & . & . & . & . & . & . & . & . & . & . & . & . & . \\[-2pt] 
. & . & 1 & . & . & . & . & . & . & . & . & . & . & . & . & . & . & a & . & . & . & . \\[-2pt] 
. & . & . & 1 & . & . & . & . & . & . & . & . & . & . & . & . & . & . & a & . & . & . \\[-2pt] 
. & . & . & . & 1 & . & b & . & . & . & . & . & . & . & . & . & . & . & . & . & . & . \\[-2pt] 
. & . & . & . & 1 & . & . & . & . & . & . & . & . & . & . & . & . & . & . & a & . & . \\[-2pt] 
. & . & . & . & . & 1 & c & . & . & . & . & . & . & . & . & . & . & . & . & . & . & . \\[-2pt] 
. & . & . & . & . & 1 & . & . & . & . & . & . & . & . & . & . & . & . & . & . & a & . \\[-2pt] 
. & . & . & . & . & . & 1 & . & . & . & . & . & . & . & . & . & . & . & . & . & . & a \\[-2pt] 
. & . & . & . & . & . & . & 1 & . & . & a & . & . & . & . & . & . & . & . & . & . & . \\[-2pt] 
. & . & . & . & . & . & . & 1 & . & . & . & . & b & . & . & . & . & . & . & . & . & . \\[-2pt] 
. & . & . & . & . & . & . & . & 1 & . & b & . & . & . & . & . & . & . & . & . & . & . \\[-2pt] 
. & . & . & . & . & . & . & . & 1 & . & . & . & . & . & . & . & . & b & . & . & . & . \\[-2pt] 
. & . & . & . & . & . & . & . & . & 1 & c & . & . & . & . & . & . & . & . & . & . & . \\[-2pt] 
. & . & . & . & . & . & . & . & . & 1 & . & . & . & . & . & . & . & . & b & . & . & . \\[-2pt] 
. & . & . & . & . & . & . & . & . & . & 1 & . & . & . & . & . & . & . & . & b & . & . \\[-2pt] 
. & . & . & . & . & . & . & . & . & . & . & 1 & c & . & . & . & . & . & . & . & . & . \\[-2pt] 
. & . & . & . & . & . & . & . & . & . & . & 1 & . & . & . & . & . & . & . & . & b & . \\[-2pt] 
. & . & . & . & . & . & . & . & . & . & . & . & 1 & . & . & . & . & . & . & . & . & b \\[-2pt] 
. & . & . & . & . & . & . & . & . & . & . & . & . & 1 & . & . & a & . & . & . & . & . \\[-2pt] 
. & . & . & . & . & . & . & . & . & . & . & . & . & 1 & . & . & . & c & . & . & . & . \\[-2pt] 
. & . & . & . & . & . & . & . & . & . & . & . & . & . & 1 & . & b & . & . & . & . & . \\[-2pt] 
. & . & . & . & . & . & . & . & . & . & . & . & . & . & 1 & . & . & . & c & . & . & . \\[-2pt] 
. & . & . & . & . & . & . & . & . & . & . & . & . & . & . & 1 & c & . & . & . & . & . \\[-2pt] 
. & . & . & . & . & . & . & . & . & . & . & . & . & . & . & 1 & . & . & . & c & . & . \\[-2pt] 
. & . & . & . & . & . & . & . & . & . & . & . & . & . & . & . & 1 & . & . & . & c & . \\[-2pt] 
. & . & . & . & . & . & . & . & . & . & . & . & . & . & . & . & . & 1 & . & . & . & c \\[-2pt] 
. & . & . & . & . & . & . & . & . & . & . & . & . & . & . & . & . & . & 1 & . & . & a \\[-2pt] 
. & . & . & . & . & . & . & . & . & . & . & . & . & . & . & . & . & . & . & 1 & . & b \\[-2pt] 
. & . & . & . & . & . & . & . & . & . & . & . & . & . & . & . & . & . & . & . & 1 & c
\end{array}
\right]
\]
\vspace{-4mm}
\caption{Rank 3, case 1: matrix of consequences in arity 10}
\label{rank3case1arity10}
\end{table}

\begin{proposition} \label{rank3case1arity10theorem}
The matrix of consequences in Table \ref{rank3case1arity10} has rank 21 or 22.
The rank is 21 for the following parameter values,
\[
(a,b,c) \; \in \; \{ \; (0,0,0), \; (-1,0,0), \; (-1,-1,-1), \; (-1,\alpha,\alpha^{-1}) \; \},
\]
where $\alpha$ is a primitive 6th root of unity ($\alpha^2 - \alpha + 1 = 0$).
The rank is 22 for all other values; since $\dim\calO(10) = 22$, 
these operads are nilpotent of index 3.
\end{proposition}

\begin{proof}
Computing the partial Smith form of this matrix takes 21 iterations;
the result is an upper left identity matrix of size 21, and a lower right column vector of size 12.
Hence the matrix has rank 21 or 22, and the rank is 21 if and only if the vector is zero.
In monic form, the 9 nonzero vector components are
\[
\begin{array}{l@{\qquad}l@{\qquad}l@{\qquad}l@{\qquad}l}
a^2(a+1), &  ba(a+1), &  ca(a+1), &  b^2(a+1), &  c^2(a+1),
\\[1mm]
b(b^2+c), &  b(bc+a), &  c(bc+a), &  c(c^2+b).
\end{array}
\]
These components generate an ideal whose deglex Gr\"obner basis ($a \prec b \prec c$) is
\[
b^2-ca, \quad c^2-ba, \quad a^2(a+1), \quad ab(a+1), \quad ac(a+1), \quad bc(a+1).
\]
This ideal is zero-dimensional, as can be seen from the first 3 elements.
Hence its zero set is finite, and the 5 points in the statement of the theorem can easily be
obtained starting from the alternative $a = 0$ or $a = -1$ (element 3).
\end{proof}

\begin{remark}
Compare Propositions \ref{rank1case1arity10theorem} and \ref{rank3case1arity10theorem}:
both have exactly 5 points which produce minimal rank.
The similarity comes from the fact that these two families of 3-parameter operads are Koszul duals, 
except for an (even) shift in the (homological) degree of the generating operation.
Markl and Remm \cite[Remark 2]{MR} state: ``If $\calP$ is a quadratic operad 
generated by an operation of arity $n$ and [homological] degree $d$, then the generating operation of 
[its Koszul dual] $\calP^!$ has the same arity but degree $-d + n - 2$''.
We have $n = 4$ and $d = 0$.
\end{remark}

\subsection{Case 2}

For rank 3, case 2, the quadratic relation matrix is
\[
\left[ \begin{array}{cccc} 
1 & 0 & a & 0 \\
0 & 1 & b & 0 \\
0 & 0 & 0 & 1
\end{array} \right]
\]
The $32 \times 22$ matrix of consequences is given in Table \ref{rank3case2arity10}; one consequence is 
generated twice but is only included once.

\begin{table}[ht]
\[
\left[
\begin{array}{cccccccccccccccccccccc}
1 & . & a & . & . & . & . & . & . & . & . & . & . & . & . & . & . & . & . & . & . & . \\[-2pt] 
1 & . & . & . & . & a & . & . & . & . & . & . & . & . & . & . & . & . & . & . & . & . \\[-2pt] 
. & 1 & b & . & . & . & . & . & . & . & . & . & . & . & . & . & . & . & . & . & . & . \\[-2pt] 
. & 1 & . & . & . & . & . & . & . & . & . & a & . & . & . & . & . & . & . & . & . & . \\[-2pt] 
. & . & 1 & . & . & . & . & . & . & . & . & . & . & a & . & . & . & . & . & . & . & . \\[-2pt] 
. & . & . & 1 & . & . & . & . & . & . & . & . & . & . & a & . & . & . & . & . & . & . \\[-2pt] 
. & . & . & 1 & . & . & . & . & . & . & . & . & . & . & . & . & . & . & . & . & . & . \\[-2pt] 
. & . & . & . & 1 & b & . & . & . & . & . & . & . & . & . & . & . & . & . & . & . & . \\[-2pt] 
. & . & . & . & 1 & . & . & . & . & . & . & . & . & . & . & a & . & . & . & . & . & . \\[-2pt] 
. & . & . & . & . & 1 & . & . & . & . & . & . & . & . & . & . & a & . & . & . & . & . \\[-2pt] 
. & . & . & . & . & . & 1 & . & . & . & . & . & . & . & . & . & . & a & . & . & . & . \\[-2pt] 
. & . & . & . & . & . & 1 & . & . & . & . & . & . & . & . & . & . & . & . & . & . & . \\[-2pt] 
. & . & . & . & . & . & . & 1 & . & a & . & . & . & . & . & . & . & . & . & . & . & . \\[-2pt] 
. & . & . & . & . & . & . & 1 & . & . & . & b & . & . & . & . & . & . & . & . & . & . \\[-2pt] 
. & . & . & . & . & . & . & . & 1 & b & . & . & . & . & . & . & . & . & . & . & . & . \\[-2pt] 
. & . & . & . & . & . & . & . & 1 & . & . & . & . & b & . & . & . & . & . & . & . & . \\[-2pt] 
. & . & . & . & . & . & . & . & . & 1 & . & . & . & . & b & . & . & . & . & . & . & . \\[-2pt] 
. & . & . & . & . & . & . & . & . & . & 1 & . & . & . & . & b & . & . & . & . & . & . \\[-2pt] 
. & . & . & . & . & . & . & . & . & . & 1 & . & . & . & . & . & . & . & . & . & . & . \\[-2pt] 
. & . & . & . & . & . & . & . & . & . & . & 1 & . & . & . & . & b & . & . & . & . & . \\[-2pt] 
. & . & . & . & . & . & . & . & . & . & . & . & 1 & . & . & . & . & b & . & . & . & . \\[-2pt] 
. & . & . & . & . & . & . & . & . & . & . & . & 1 & . & . & . & . & . & . & . & . & . \\[-2pt] 
. & . & . & . & . & . & . & . & . & . & . & . & . & 1 & . & a & . & . & . & . & . & . \\[-2pt] 
. & . & . & . & . & . & . & . & . & . & . & . & . & . & 1 & b & . & . & . & . & . & . \\[-2pt] 
. & . & . & . & . & . & . & . & . & . & . & . & . & . & . & . & 1 & . & . & . & . & . \\[-2pt] 
. & . & . & . & . & . & . & . & . & . & . & . & . & . & . & . & . & 1 & . & . & . & . \\[-2pt] 
. & . & . & . & . & . & . & . & . & . & . & . & . & . & . & . & . & . & 1 & . & a & . \\[-2pt] 
. & . & . & . & . & . & . & . & . & . & . & . & . & . & . & . & . & . & 1 & . & . & . \\[-2pt] 
. & . & . & . & . & . & . & . & . & . & . & . & . & . & . & . & . & . & . & 1 & b & . \\[-2pt] 
. & . & . & . & . & . & . & . & . & . & . & . & . & . & . & . & . & . & . & 1 & . & . \\[-2pt] 
. & . & . & . & . & . & . & . & . & . & . & . & . & . & . & . & . & . & . & . & 1 & . \\[-2pt] 
. & . & . & . & . & . & . & . & . & . & . & . & . & . & . & . & . & . & . & . & . & 1
\end{array}
\right]
\]
\vspace{-4mm}
\caption{Rank 3, case 2: matrix of consequences in arity 10}
\label{rank3case2arity10}
\end{table}

\begin{proposition}
The matrix of consequences for rank 3, case 2 has rank 21 if and only if $a = b = 0$; for all other
values of the parameters, the matrix has rank 22 and the operad is nilpotent of index 3.
\end{proposition}

\begin{proof}
Computing a partial Smith form of this matrix requires 21 iterations;
we obtain an upper left identity matrix of size 21, and a lower right block which is a column vector of size 11.
The matrix has rank 21 if and only if the vector is zero.
Four entries of the vector are zero, and the monic forms of the remaining seven are
$a$,  $b$,  $ab$,  $a^3$,  $a^2b$,  $ab^2$,  $b(b^2+a)$.
These elements generate an ideal with Gr\"obner basis $\{ a, b \}$, and its zero set is the single point 
$\{ (0,0) \}$.
\end{proof}

\subsection{Case 3}

For rank 3, case 3, the quadratic relation matrix is
\[
\left[ \begin{array}{cccc} 
1 & a & 0 & 0 \\
0 & 0 & 1 & 0 \\
0 & 0 & 0 & 1
\end{array} \right]
\]
There is only one parameter; the matrix entries belong to $\mathbb{F}[a]$, and so we can
compute the HNF to understand how the rank depends on the parameter.

\begin{proposition}
The matrix of consequences for rank 3, case 3 has rank 21 if $a = 0$ and rank 22 otherwise 
(in which case the operad is nilpotent of index 3).
\end{proposition}

\begin{proof}
The $28 \times 22$ matrix of consequences has leading 1s in these 28 positions,
  \[
  \begin{array}{l}
  (1,1), (2,1), (3,2), (4,3), (5,3), (6,4), (7,4), (8,5), (9,6), (10,6), (11,7), (12,7), 
  \\
  (13,8), (14,10), (15,11), (16,12), (17,13), (18,14), (19,14), (20,15), (21,16), 
  \\
  (22,17), (23,18), (24,19), (25,19), (26,20), (27,21), (28,22) 
  \end{array} 
  \]
and the parameter $a$ in these 11 positions,
  \[
  \begin{array}{l}
  (1,2), (2,5), (3,8), (4,9), (6,10), (8,11), (9,12), (11,13), (13,9), (18,15), (24,20);
  \end{array} 
  \]
all other entries are 0.
The HNF is $I_{22}$ except for $a$ in position $(9,9)$.
\end{proof}

\subsection{Case 4}

For rank 3, case 4, the quadratic relation matrix is
\[
\left[ \begin{array}{cccc} 
0 & 1 & 0 & 0 \\
0 & 0 & 1 & 0 \\
0 & 0 & 0 & 1
\end{array} \right]
\]

\begin{proposition}
The matrix of consequences for rank 3, case 4 has rank 21.
\end{proposition}

\begin{proof}
The matrix of consequences has size $21 \times 22$; the first column is zero, and 
the rest is an identity matrix of size 21.
This matrix is already in RCF.
\end{proof}

\subsection{Summary for quadratic relation rank 3}

In what follows $\alpha$ is a primitive 6th root of unity ($\alpha^2 - \alpha + 1 = 0$).

\begin{theorem}
For a 3-dimensional space of quadratic relations in arity 7, the rank $r$ of the cubic consequences is either 21 or 22.
The following is a complete list of triples of relations for which the rank achieves the minimal value of 21.
  \begin{align*}
  1a) \quad 
  \rho_1 &= 
  \phi \circ_1 \phi,
  \quad
  \rho_2 =
  \phi \circ_2 \phi,
  \quad
  \rho_3 =
  \phi \circ_3 \phi,
  \\
  1b) \quad 
  \rho_1 &= 
  \phi \circ_1 \phi - \phi \circ_4 \phi,
  \quad
  \rho_2 =
  \phi \circ_2 \phi,
  \quad
  \rho_3 =
  \phi \circ_3 \phi,
  \\
  1c) \quad 
  \rho_1 &= 
  \phi \circ_1 \phi - \phi \circ_4 \phi,
  \quad
  \rho_2 =
  \phi \circ_2 \phi - \phi \circ_4 \phi,
  \quad
  \rho_3 =
  \phi \circ_3 \phi - \phi \circ_4 \phi,
  \\
  1d) \quad 
  \rho_1 &= 
  \phi \circ_1 \phi - \phi \circ_4 \phi,
  \quad
  \rho_2 =
  \phi \circ_2 \phi + \alpha \phi \circ_4 \phi,
  \quad
  \rho_3 =
  \phi \circ_3 \phi + \alpha^{-1} \phi \circ_4 \phi,
  \\
  2) \quad 
  \rho_1 &= 
  \phi \circ_1 \phi,
  \quad
  \rho_2 =
  \phi \circ_2 \phi,
  \quad
  \rho_3 =
  \phi \circ_4 \phi,
  \\
  3) \quad 
  \rho_1 &= 
  \phi \circ_1 \phi,
  \quad
  \rho_2 =
  \phi \circ_3 \phi,
  \quad
  \rho_3 =
  \phi \circ_4 \phi,
  \\
  4) \quad 
  \rho_1 &= 
  \phi \circ_2 \phi,
  \quad
  \rho_2 =
  \phi \circ_3 \phi,
  \quad
  \rho_3 =
  \phi \circ_4 \phi.
  \end{align*}
In all other cases, the rank is 22 and the operad is nilpotent of index 3.
\end{theorem}


\section*{Acknowledgements}

Murray Bremner was partially supported by a Discovery Grant from NSERC, 
the Natural Sciences and Engineering Research Council of Canada.
His visit to Cape Town in November-December 2015 was partially supported by a grant from 
NRF, the National Research Foundation of South Africa.
He thanks the Department of Mathematics and Applied Mathematics at the University of Cape Town 
for its hospitality.



\section*{References}

\begin{thebibliography}{99}

\bibitem{BD}
\textsc{M. R. Bremner, V. Dotsenko}:
\emph{Algebraic Operads: An Algorithmic Companion}.
CRC Press, Boca Raton, 2016 (to appear).

\bibitem{Cox}
\textsc{D. A. Cox, J. Little, D. O'Shea}:
\emph{Using Algebraic Geometry}.
Second edition.
Springer, New York, 2005.

\bibitem{Gnedbaye}
\textsc{A. V. Gnedbaye}:
Op\'erades des alg\`ebres ($k{+}1$)-aires. 
\emph{Operads: Proceedings of Renaissance Conferences (Hartford, CT/Luminy, 1995)}, 
83--113. 
Contemp. Math., 202. 
Amer. Math. Soc., Providence, RI, 1997. 

\bibitem{GKP}
\textsc{R. L. Graham, D. E. Knuth, O. Patashnik}:
\emph{Concrete Mathematics: A Foundation for Computer Science}.
Second edition. 
Addison-Wesley, Reading, 1994.

\bibitem{LV}
\textsc{J.-L. Loday, B. Vallette}:
\emph{Algebraic Operads}.
Grundlehren der Mathematischen Wissenschaften, 346.
Springer, Heidelberg, 2012.

\bibitem{MR}
\textsc{M. Markl, E. Remm}:
(Non-)Koszulness of operads for $n$-ary algebras, galgalim and other curiosities.
\emph{J. Homotopy Relat. Struct.}
Published online: 15 October 2014.
DOI: 10.1007/s40062-014-0090-7.

\end{thebibliography}
\end{document}